\documentclass[12pt,letterpaper,reqno]{amsart}

\usepackage{euscript}
\usepackage{amsthm}
\usepackage{amssymb,amscd}
\usepackage{amsmath}
\usepackage{graphicx}
\usepackage{color}

\usepackage{lmodern}
\usepackage[T1]{fontenc}
\usepackage{textcomp}

\usepackage[text={6.3in, 8.5in}, centering]{geometry}

\usepackage{microtype}

\newtheorem*{remark}{Remark}
\newtheorem{theorem}{Theorem}[section]
\newtheorem{proposition}[theorem]{Proposition}

\newtheorem{corollary}[theorem]{Corollary}
\newtheorem{lemma}[theorem]{Lemma}

\newtheorem*{main}{Main result}
\newtheorem{mainthm}{Main Theorem}

\newcommand{\Z}{\mathbb{Z}}
\newcommand{\E}{\mathbb{E}}
\newcommand{\R}{\mathbb{R}}

\newcommand{\N}{\mathbb{N}}

\newcommand{\T}{\mathbb{T}}

\newcommand{\cG}{\mathcal{G}}
\newcommand{\bbS}{\mathcal{S}}
\newcommand{\bbU}{\mathcal{U}}

\newcommand{\bbA}{\mathcal{A}}
\newcommand{\bbB}{\mathcal{B}}
\newcommand{\bbC}{\mathcal{C}}
\newcommand{\bbD}{\mathcal{D}}
\newcommand{\bbL}{\mathcal{L}}

\newcommand{\bara}{a^*}
\newcommand{\barA}{A^*}

\DeclareMathOperator*{\argmin}{arg\,min}
\DeclareMathOperator{\supp}{supp}
\renewcommand{\mod}{\,\mathrm{mod}\,}
\DeclareMathOperator{\diam}{diam}
\DeclareMathOperator{\diag}{diag}
\DeclareMathOperator{\conv}{conv}

\begin{document}

\title[Minimizers and viscosity solutions]{Hyperbolicity of minimizers and regularity of viscosity solutions for a random Hamilton-Jacobi equation}
\author{Konstantin Khanin}
\address{University of Toronto}
\email{khanin@math.toronto.edu}
\author{Ke Zhang}
\address{University of Toronto}
\email{kzhang@math.toronto.edu}

\subjclass[2010]{35R60; 37J50,37H99,37L55,37D25,76F20}

\begin{abstract}
	We show that for a large class of randomly kicked Hamilton-Jacobi equations, the unique global minimizer is almost surely hyperbolic. Furthermore, we prove that the unique forward and backward viscosity solutions, though in general only Lipshitz, are smooth in a neighborhood of the global minimizer. Related results in the one-dimensional case were obtained by E, Khanin, Mazel and Sinai in \cite{EKMS00}. However the methods in the above paper is purely one-dimensional and cannot be extended to the case of higher dimensions. Here we develop a completely different approach. 
\end{abstract}

\maketitle

\section{Introduction}
\label{sec:introduction}

In this paper we consider the periodic inviscid Burger's equation 
\begin{equation}
	\label{eq:burgers}
	\partial_t u + (u\cdot \nabla)u = f^\omega(y,t), \quad y\in \T^d, t\in \R,
\end{equation}
where $f^\omega(y,t)= -\nabla F^\omega(y,t)$ is a random force given by the potential $F^\omega$. 
Burger's equation plays an important role in a large variety of mathematical and physical problems. It appears in description of a whole range of extended
dissipative systems featuring nonequilibrium turbulent processes, from microscales of
condensed matter and statistical physics to macroscale cosmological evolution (see \cite{BK} for related 
references).
The inviscid equation can be viewed as the limit of the viscous Burger's equation
\[
	\partial_t u + (u\cdot \nabla)u = \nu \Delta u + f^\omega(y,t), \quad y\in \T^d, t\in \R,
\]
as the viscosity $\nu$ approaches zero. The solutions
obtained in the vanishing viscosity limit are called viscosity solutions. For more details
on the viscosity limit we refer to \cite{GIKP05} and the references therein.  

The theory of (\ref{eq:burgers}) is developed by E, Khanin, Mazel and Sinai in \cite{EKMS00}, for the one dimensional configuration space ($d=1$) and the ``white noise'' potentials 
\begin{equation}
	\label{eq:whitenoise}
	F^\omega(y,t) = \sum_{i=1}^M F_i(y,t) = \sum_{i=1}^M F_i(y) \dot{W}_i(t),
\end{equation}
where $F_i:\T^d \to \R$ are smooth functions, and $\dot{W}_i$ are independent white noises. 

In \cite{IK03}, Iturriaga and Khanin considered the both the ``white noise'' potential \eqref{eq:burgers} and the  ``kicked'' potential 
\begin{equation}\label{eq:kicked} F^\omega(y,t) = \sum_{j\in \Z}F^\omega_j(y)\delta(t-j),
\end{equation}
where $F^\omega_j$ are chosen independently from the same distribution, and $\delta(\cdot)$ is the delta function. The potential (\ref{eq:kicked}) can be considered as a discrete time version of (\ref{eq:whitenoise}). The theory of (\ref{eq:kicked}) runs parallel to the theory of (\ref{eq:whitenoise}), and has the advantage of being technically simpler. 

We shall consider solutions of the type $u=\nabla \phi$, which converts (\ref{eq:burgers}) to the Hamilton-Jacobi equation 
\begin{equation}\label{eq:hj} \partial_t\phi + \frac12(\nabla \phi)^2 + F^\omega(y,t) =0.
\end{equation}
Hence, for the solutions of interest, it is equivalent to study the viscosity solutions of (\ref{eq:hj}). Moreover, for each $b\in \R^d$, we may consider solutions of (\ref{eq:hj}) with $\int \nabla \phi(y,t) dy = b$, as the vector $b$ is invariant under the evolution. 

The study of these solutions is closely related with the concept of minimizing orbits in Lagrangian systems. More precisely, for each $b\in \R$, $s<t$, $x,x'\in \T^d$, we define the minimal action function by 
\begin{equation}
	\label{eq:b-var} A_{s,t}^{\omega,b}(x,x') = \inf \int_s^t  \frac12 (\dot \zeta (\tau))^2 - \dot\zeta \cdot b - F^\omega(\zeta(\tau),\tau)d\tau,
\end{equation}
where the infimum is taken over all absolutely continuous curves $\zeta:[s,t]\to \T^d$ with $\zeta(s)=x$ and $\zeta(t)=x'$. A curve $\gamma:I\to \T^d$ is called a \emph{minimizer} if the infimum in $A_{s,t}^{\omega,b}(\gamma(s), \gamma(t))$ is achieved at $\gamma|[s,t]$ for all $[s,t]\subset I \subset \R$. For the cases $I=[t_0,\infty)$, $I=(-\infty, t_0]$ or $I=\R$, the curve $\gamma$ is called a \emph{forward minimizer}, a \emph{backward minimizer} or a \emph{global minimizer}, respectively. 

Deferring the precise definitions to the next section, we roughly reinterpret the main results of \cite{EKMS00} (using the point of view of \cite{IK03}) as follows. For $d=1$, $b\in \R$, under some nondegeneracy conditions, the following hold almost surely. 
\begin{enumerate}
	\item[C1.] (``One force, one solution principle'') There exists a unique global viscosity solution $\psi(x,t)$ on the set $\T^d\times (-\infty, \infty)$, up to a constant translation;
	\item[C2.] For Lebesgue a.e. $x\in \T^d$, there exist unique forward and backward minimizers;
	\item[C3.]  There exists a unique global minimizer supporting a unique invariant measure for the random Lagrangian flow. 
	\item[C4.] The invariant measure from C3  has no zero Lyapunov exponents (and hence is hyperbolic). 
	\item[C5.] The unique viscosity solutions is \emph{smooth} in a neighborhood of the global minimizer. Furthermore, the graph of its gradient $\{(y, \nabla \psi)\}$ is equal to the local unstable manifold of the global minimizer.  
\end{enumerate}

In \cite{IK03}, conclusions C1-C3 were generalized to arbitrary dimensions. The approach in \cite{IK03} is variational, and it is the starting point of this paper.  It is related to Aubry-Mather theory and weak KAM theory (see for example, \cite{Mather91}, \cite{Mather93}, \cite{fat08}). 

We describe our main result roughly as follows (see Main Theorem~\ref{hyperbolic}, Main Theorem~\ref{smooth-vis} for accurate statements):
\begin{main}
	For the ``kicked'' potential, conclusion C4 and C5 hold in arbitrary dimensions. In other words, the unique global minimizer is hyperbolic, and the viscosity solutions are smooth near the global minimizer. 
\end{main}

Hyperbolicity of the global minimizer is a  property of crucial importance
in the analysis of dynamics defined by the Burgers equation. In our case, it enables the application of general theory of non-uniform hyperbolicity to establish regularity. 
\begin{remark}
	As mentioned before, it is natural to expect that the methods applied to the ``kicked'' case in this paper should apply to the ``white noise'' case. We choose to convey the ideas of our method in the  technically simpler ``kicked'' case, and treat the ``white noise'' case in a later work. 
\end{remark}

The proofs of C4 and C5 in the one-dimensional case (in \cite{EKMS00}) depends strongly on one-dimensionality and cannot be extended to higher dimensions. To prove C4 for arbitrary dimensions, we devise a completely different strategy. The main new ingredient is the use of the Green bundles (see \cite{Gre58}). These bundles have been useful in proving uniform hyperbolicity for Hamiltonian flows, and have recently been used to study Lyapunov exponents, due to the work of Arnaud (\cite{Arn12}, \cite{Arnaud2013}). It was known for some time that Green bundles are related to the hyperbolic properties of the global minimizer. The main achievement in the present paper is a quantitative analysis which can be carried out almost surely. In particular, we are proving a quantitative version of Alexandrov's theorem which is interesting on its own.

When the global minimizer is nonuniformly hyperbolic, there is no clear-cut relation between the viscosity solutions and the stable/unstable manifolds.  In general, the viscosity solutions, as a product of the variational method, may bare no direct relations to the stable/unstable manifolds. Our proof of C5 from C4 uses precise information  of the variational problem (more than what's needed to prove C4), and requires a careful adaptation of the nonuniform hyperbolic theory.  We would like to mention that for nonrandom systems, under the easier assumption that the global minimizer is \emph{uniformly} hyperbolic, an analogous result is known (see \cite{Ber07}). 

The outline of the paper is as follows. In section~\ref{sec:main}, we introduce the notations, recall the results of Iturriaga and Khanin in \cite{IK03}, and formulate the main technical statements. In section~\ref{sec:visc-solut-minim}, we describe the variational set up, and formulate some statements in variational analysis. The proofs of these statements are deferred to the end of the paper. In section \ref{sec:green-bundles}, we define the Green bundles, and establish their connections to the variational problem. In section \ref{sec:nonzero-exp}, we show that the transversality of the Green bundles imply nonzero exponents, proving Main Theorem~\ref{hyperbolic}. Sections \ref{sec:green-bundles} and \ref{sec:nonzero-exp} make use of the results of Arnaud in \cite{Arn12} and \cite{Arnaud2013}. In section \ref{sec:loc-smooth} and \ref{sec:smoothness}, we prove Main Theorem~\ref{smooth-vis} using variational arguments and Pesin's theory. In the last three sections, the technical statements from section~\ref{sec:visc-solut-minim} are proved.

\section{Formulation of the main results}
\label{sec:main}

We restrict ourselves to the case of ``kicked'' potentials
$$F(y,t)=\sum_{j\in \Z}F_j(y)\delta(t-j).$$
Here we assume that the potentials $F_j$ are chosen independently from a distribution $\chi \in P(C^{2+\alpha}(\T^d))$, $0<\alpha\le 1$ (some of the results hold under weaker regularity assumptions).

Since we will be considering the solutions $\phi$ of the type $\int \nabla \phi dy =b$, we write $\phi(y,t) = b\cdot y + \psi(y,t)$, where $\int \nabla \psi dy =0$. It's easy to see that $\psi$ is a viscosity solution of the Hamilton-Jacobi equation 
\begin{equation}\label{eq:b-hj}  \partial_t \psi(x,t) + H^b(\nabla \psi(x,t), t)=0,
\end{equation}
with the Hamiltonian $H^b(x,p,t)=\frac12 (p +b)^2 + F^\omega(x,t)$. The corresponding Lagrangian is given by $L(x,v,t) = \frac12 v^2 - b\cdot v - F^\omega(x, t)$, which is the same Lagrangian in \eqref{eq:b-var}. According to the Lax-Oleinik variational principle, given $x,x'\in \T^d$ and $s<t$, we have 
$$ \psi(x',t)=\inf_{x\in \T^d}\left\{\psi(x,s) + A_{s,t}^{\omega, b}(x,x')  \right\}.$$
Note that $F^\omega(y,t)=0$ for all $t\notin \Z$. As a consequence, any curve $\zeta$ realizing the minimum in the definition of $A_{s,t}^{\omega,b}$ must be linear between integer values of $\zeta$. In this sense, the viscosity solutions are completely determined by their values at $t\in \Z$. 

For $b\in \R^d$, $m,n\in \Z$, $m<n$, we define the discrete version of the action function by 
\begin{equation}
	\label{eq:ad-act}
	A_{m.n}^{\omega,b}(x,x')=\inf \left\{ \sum_{i=m}^{n-1}\frac12(\tilde{x}_{i+1}-\tilde{x}_i)^2-b\cdot(\tilde{x}_n-\tilde{x}_m) - \sum_{i=m}^{n-1}F_i^\omega(\tilde{x}_i) \right \},
\end{equation}
where the infimum is taken over all $(\tilde{x}_j)_{j=m}^n$, $\tilde{x}_j\in \R^d$ such that $\tilde{x}_n=x$ and $\tilde{x}_m=x'$ ($\text{mod } \Z^d$). The sequence $(\tilde{x}_j)_{j=m}^n$ corresponds to the lift of the curve $\zeta$ in (\ref{eq:b-var}), and we call it a \emph{configuration} following the language of twist diffeomorphisms. 
Throughout the paper, we may drop the supscript $\omega$ and $b$ when there is no risk of confusion. 

The solution $\psi(x,n)$ is closely related to the family of maps $\Phi_j^\omega:\T^d\times \R^d \to \T^d \times \R^d$
\begin{equation}
	\label{eq:twistmaps}
	\Phi_j^\omega:
	\begin{bmatrix}
		x_j \\ v_j
	\end{bmatrix}
	\mapsto
	\begin{bmatrix}
		x_{j+1} \\ v_{j+1} 
	\end{bmatrix} 
	=
	\begin{bmatrix}
		x_j + v_j - \nabla F_j^\omega(x_j)  \mod \Z^d.
		\\
		v_j - \nabla F_j^\omega(x_j)
	\end{bmatrix},
\end{equation}
The maps belong to the so-called \emph{standard family}, and are examples of symplectic, exact and monotonically twist diffeomorphisms.  For $m,n\in \Z$, $m<n$, denote
$$ \Phi^\omega_{m, n}(x,v)=\Phi^\omega_{n-1}\circ \cdots \circ \Phi^\omega_{m}(x,v). $$
For any $n\in \Z$ and $(x_n,v_n)\in \T^d\times \R^d$, we define $(x_j,v_j)=\Phi_{n,j}(x_n,v_n)$ if $j>n$ and $(x_j,v_j)=(\Phi_{j,n})^{-1}(x_n,v_n)$ if $j<n$. We call $(x_j,v_j)_{j\ge n}$ the forward orbit of $(x_n, v_n)$, $(x_j,v_j)_{j\le n}$ the backward orbit of $(x_n, v_n)$, and $(x_j,v_j)_{j\in \Z}$ the full orbit of $(x_n,v_n)$.

It is easy to see that the orbit of the maps $\{\Phi^\omega_j\}$ is a discretization of the Euler-Lagrange flow. It follows that any minimizer of \eqref{eq:ad-act} corresponds to an orbit of \eqref{eq:twistmaps}. Indeed, if $v_j = \tilde{x}_{j+1} - \tilde{x}_j + \nabla F_j^\omega(\tilde{x}_j)$ and $x_j = \tilde{x}_j \mod \Z^d$, then $\Phi^\omega_j(x_j, v_j) = (x_{j+1}, v_{j+1})$. Conversely, any orbit $\{(x_j, v_j)_{j=m}^{n-1}$ defines a configuration $\{\tilde{x}_j\}_{j=m}^{n}$ which is unique up to integer translation.

The following assumptions on the probability space $P(C^2(\T^d))$ were introduced in \cite{IK03}.
\begin{itemize}
	\item[] \emph{Assumption 1.} For any $y\in\T^d$, there exists $G_y\in \supp P$ s.t. $G_y$ has a maximum at $y$ and that there exists $\delta >0$ such that 
	$$ G_y(y)-G(x)\ge \delta |y-x|^2. $$
	\item[] \emph{Assumption 2.} $0\in \supp P$. 
	\item[] \emph{Assumption 3.} There exists $G\in \supp P$ such that $G$ has a unique maximum. 
\end{itemize}

We define the backward Lax-Oleinik operator $K_{m,n}^{\omega,b}:C(\T^d)\to C(\T^d)$, for $m<n$, $m,n\in \Z$ by the following expression: 
\begin{equation}
	\label{eq:LO-minus}
	K_{m, n}^{\omega,b}\varphi(x)=\inf_{x_m\in \T^d}\{\varphi(x_m) + A_{m,n}^{\omega,b}(x_m, x)\}. 
\end{equation}
The family of functions $K_{m,n}^{\omega, b}\varphi(x)$, $n > m$ is precisely the solution of the Hamilton-Jacobi equation \eqref{eq:b-hj} at integer times $t =n$, with initial value $\varphi$ at $t = m$. The following theorem establishes the existence and uniqueness of the solution on the time interval $(-\infty, n_0]$. 
\begin{theorem}\cite{IK03}\label{backward-min}
\begin{enumerate}
	\item Assume that assumption 1 or 2 holds.  For  each $n_0\in \Z$, for a.e. $\omega\in \Omega$, we have the following statements.
	\begin{enumerate}
		\item There exists a Lipshitz function $\psi^-(x,n)$, $n\le n_0$, such that for any $m<n\le n_0$,
		$$ K_{m,n}^{\omega,b}\psi^-(x,m)=\psi^-(x,n).$$
		\item For any $\psi\in C(\T)$ and $n\le n_0$, we have
		$$ \lim_{m\to -\infty}\inf_{C\in \R}\|K^{\omega,b}_{m,n}\psi(x)-\psi^-(x,n)-C\| =0. $$
		\item For $n\le n_0$ and Lebesgue a.e. $x\in \T^d$, the gradient $\nabla \psi^-(x,n)$ exists. Denote $x_n^-=x$ and $v_n^-=\nabla\psi^-(x_n^-, n)+b$;  then $(x_j^-, v_j^-)= (\Phi_{j,n})^{-1}(x_n^-, v_n^-)$, $j\le n$ is a backward minimizer.
	\end{enumerate}
	\item Assume that assumption 3 holds. Then the conclusions for the first case hold for $b=0$.  
\end{enumerate}
\end{theorem}

Item 1(b) of Theorem~\ref{backward-min} implies that almost surely, $\psi^-(x,n)$ for $- \infty < n \le n_0$ is unique up to an additive constant. Therefore solutions corresponding different time intervals $(-\infty, n_1]$ and $(-\infty, n_2]$ coincide on their common domain. Then almost surely, there is a unique invariant family defined on the interval $(-\infty, \infty)$.

Similar theorems hold for the forward minimizers. For $\varphi\in C(\T^d)$, $m,n\in \Z$, $m<n$, we define the forward Lax-Oleinik operator as follows: 
\begin{equation}
	\label{eq:LO-plus}
	\tilde{K}_{m.n}^{\omega,b}\varphi(x)=\sup_{x_n\in \T^d}\{\varphi(x_n) - A_{m,n}^{\omega,b}(x, x_n)\}. 
\end{equation}
Under the same assumptions as in Theorem~\ref{backward-min}, there exists a Lipshitz function $\psi^+(x,n)$, $n\ge n_0$, such that 
$$ \tilde{K}_{m,n}^{\omega,b} \psi^+(x,n)=\psi^+(x,m).$$
For $n\ge n_0$ and Lebesgue a.e. $x\in \T^d$, $v_n^+=-\nabla \psi^+(x,n)+b$ exists. Write $x_n^+=x$; we have that $(x_j^+, v_j^+)=\Phi^\omega_{j,n} (x_n^+, v_n^+)$ is a forward minimizer. 

We further reduce the choice of our potential to the one generated by a finite family of smooth potentials, multiplied by i.i.d. random vectors. These potentials emulate the behaviour of the ``white noise'' case. 
\begin{itemize}
	\item[] \emph{Assumption 4.} Assume that 
	\begin{equation}
		\label{eq:F-omega}
		F^\omega_j(x)=\sum_{i=1}^M \xi_j^i(\omega)F_i(x),    
	\end{equation}
	where $F_i:\T^d\to \R$ are smooth functions, and the vectors $\xi_j(\omega)=(\xi_j^i(\omega))_{i=1}^M$ are identically distributed vectors in $\R^M$ with an absolutely continuous distribution. 
\end{itemize}

We have the following theorem from \cite{IK03}.
\begin{theorem}\label{global-min}\cite{IK03}
\begin{enumerate}
	\item Assume that assumption 4 and one of assumptions 1 and 2 hold. If 
	\begin{equation}
		\label{eq:Fi}
		(F_1, \cdots F_M): \T^d \to \R^M
	\end{equation}
	is one-to-one, then for all $b\in \R^d$ and a.e. $\omega$ there exists a unique $(x^\omega_0, v^\omega_0)\in \T^d\times \R^d$, such that the full orbit of $(x^\omega_0, v^\omega_0)$ is a global minimizer. 
	\item The same conclusion is valid if assumption 3 holds and $b=0$. 
\end{enumerate}
\end{theorem}

\begin{remark}
	\begin{itemize}
		\item Assuming assumption 4, assumption 2 is satisfied if the distribution for the random vector $\xi_j(\omega)$ contains $0$ in its support. In particular, this would be the case when $\xi_j$ is Gaussian, which is the discrete counterpart to the white noise case. 
		\item Assumption 1 can be satisfied if the functions $F_i$ and the distribution is well chosen, for example, if $d =1$, $M=2$, $F_1 = \cos(2\pi x)$, $F_2 = \sin(2\pi x)$, and $\xi_j$ is fully supported on $\R^2$.
		\item  Assumption 3 holds if the family $(F_1, \cdots, F_M)$ is chosen in a generic way.  
	\end{itemize}
\end{remark}

As the random potential is generated by a stationary random process, the time shift $\theta^m$ is a metric isomorphism of the probability space $\Omega$ satisfying 
$$ F^\omega(y, n+m)=F^{\theta^m \omega}(y,n), \quad m\in \Z.$$
The family of maps $\Phi^\omega_j$ then defines a non-random transformation $\hat{\Phi}$ on the space $\T^d \times \R^d \times \Omega$ given by 
\begin{equation}
	\label{eq:random-map}
	\hat{\Phi}(x,v,\omega)=(\Phi^\omega_0(x,v), \theta \omega).   
\end{equation}

Let $(x^\omega_j, v^\omega_j)$ be the global minimizer in Theorem~\ref{global-min}. We have that the probability measure 
$$\nu(d(x,v),d\omega)= \delta_{x^\omega_0, v^\omega_0}(d(x,v))P(d\omega)$$ 
is invariant and ergodic under the transformation \eqref{eq:random-map}. The map $D\Phi^\omega_0:\T^d\times \R^d \to Sp(d)$ defines a cocycle over the transformation~\eqref{eq:random-map}, where $Sp(d)$ is the group of all $2d\times 2d$ symplectic matrices. Under the first part of assumption 5 below, the Lyapunov exponents for this cocycle are well defined.  Denote them by  $\chi_1(\nu),\cdots, \chi_{2d}(\nu)$. Due to the symplectic nature of the cocycle we have 
$$ \chi_1(\nu)\le \cdots \le \chi_d(\nu) \le 0 \le \chi_{d+1}(\nu)\le\cdots\le \chi_{2d}(\nu).$$
Moreover, $\chi_i = - \chi_{2d-i+1}$. 

To show the Lyapunov exponents are nonzero, we require an additional assumption. Let $\rho$ be the probability density for the vectors $\xi_j \in \R^M$.  
\begin{itemize}
	\item[] \emph{Assumption 5.}  Suppose assumption 4 holds, and in addition:
	\begin{itemize}
		\item $E(|\xi_j|)=\int_{\R^M} |c|\rho(c)dc < \infty$. 
		\item For every $1\le i \le M$, there exists non-negative functions $\rho_i \in L^\infty(\R)$ and $\hat{\rho}_i\in L^1(\R^{M-1})$ such that 
		$$ \rho(c)\le \rho_i(c_i) \hat{\rho}_i(\hat{c}), $$
		where $c=(c_1, \cdots, c_M)$, $\hat{c}_i =(c_1, \cdots, c_{i-1}, c_{i+1}, \cdots, c_M)$. 
	\end{itemize}
\end{itemize}
\begin{remark}
	Assumption 5 is rather mild. For example, it is satisfied if $\xi^1_j, \cdots, \xi^M_j$ are independent random variables with bounded densities and finite mean. In fact, independence is not essential, we only need to avoid the case when the density $\rho$ is degenerate in certain directions. 
\end{remark}
We shall replace the one-to-one condition from Theorem~\ref{global-min} with a stronger condition, requiring the map (\ref{eq:Fi}) is an embedding. 
\begin{mainthm}\label{hyperbolic}
\begin{enumerate}
	\item   Assume that assumption 5 and one of assumptions 1 or 2 holds. Assume in addition that the map \eqref{eq:Fi} is an embedding. Then for all $b\in \R^d$, for a.e. $\omega$, the Lyapunov exponents of $\nu$ satisfy
	$$ \chi_d(\nu)<0<\chi_{d+1}(\nu).$$
	\item For the case $b=0$, assumption 1 or 2 can be replaced with the weaker assumption 3. The same conclusion holds. 
\end{enumerate}
\end{mainthm}

A corollary of Main Theorem~\ref{hyperbolic} is that for a.e. $\omega$, the orbit $(x_k^\omega, v_k^\omega)_{k\in \Z}$ is non-uniformly hyperbolic (see \cite{BP07}). With a slightly stronger regularity assumption on the map (\ref{eq:Fi}), there exists local unstable manifold $W^u(x_k,v_k)$ and stable manifold $W^s(x_k,v_k)$. 

Our next theorem states that, near $(x_k,v_k)$, $W^u$ and $W^s$ coincide with the sets  $W^-_k=\{(x, \nabla_x\psi^-(x,k))\}$ and $W^+_k=\{(x,\nabla_x\psi^+(x, k))+b\}$. The functions $\psi^\pm(\cdot,k)$ are only Lipshitz and  $W^\pm_k$ are only defined above a full (Lebesgue) measure of $x$. Since $W^u$ and $W^s$ are smooth manifolds, our theorem states that $\psi^\pm$ are in fact smooth in a neighbourhood of $x_k$.

\begin{mainthm}\label{smooth-vis}
\begin{enumerate}
	\item Assume that assumptions 4 and 5 hold, and one of assumptions 1 or 2 holds. Assume in addition that, for $0<\alpha\le 1$, the map (\ref{eq:Fi}) is a $C^{2+\alpha}$ embedding. Then for all $b\in \R^d$, for a.e.$\omega$, there exists a neighbourhood $V(\omega)$ of $(x_0,v_0)$ such that $\psi^\pm(x,0)$ are $C^2$ smooth in $V(\omega)$, and 
	$$W_0^-\cap V = W^u(x_0,v_0)\cap V, \quad W_0^+\cap V = W^s(x_0,v_0)\cap V.$$

	\item For the case $b=0$, assumption 1 or 2 can be replaced with the weaker assumption 3. The same conclusions hold. 
\end{enumerate}   
\end{mainthm}

\section{Viscosity solutions and the minimizers}
\label{sec:visc-solut-minim}

In this section we will first deduce some useful properties of the action functional, and introduce a variational problem closely related to the global minimizer. The derivation of the variational problem mostly follow \cite{IK03}. 

We say that a  function $f:\T^d\to \R$ is $C-$semi-concave on $\T^d$ if for any $x\in \T^d$, there exists a linear form $l_x:\R^d\to \R$ such that for any $y\in \T^d$,
$$ f(y)-f(x)\le l_x(y-x) + \frac{C}2 d(x,y)^2.$$
Here $d(x,y)$ is understood as the distance on the torus, and the vector $y-x$ is interpreted as any vector from $x$ to $y$ on the torus.  In what follows we often call the slope of $l_x$ a subdifferential. 

\begin{lemma}[\cite{fat08}, Proposition 4.7.3]\label{Lipshitz}
If $f$ is  continuous and $C-$semi-concave on $\T^d$, then there exists a unique $C'>0$ depending only on $C$ such that $f$ is $C'-$Lipshitz. 
\end{lemma}

Let $C_j^\omega=\|F_j\|_{C^2}$. The action function $A_{m,n}^{\omega,b}$ has the following properties. 

\begin{lemma}\label{properties} For any $b\in \R^d$, the action function $A^{\omega,b}_{m,n}$ satisfy the following properties:
\begin{itemize}
	\item For any $m<k<n$, $A^{\omega,b}_{m,n}(x, x')=\min_{x_k\in \T^d}\{A_{m,n}^{\omega,b}(x,x_k)+ A_{m,n}^{\omega,b}(x_k, x') \}$.
	\item If $(x_j, v_j)_{j=m}^n$  is a minimizer, then $A^{\omega,b}_{m,n}(x_m,x_n)=\sum_{j=m}^{n-1}A^{\omega,b}_{j,j+1}(x_j, x_{j+1})$. 
	\item The function $A_{m,n}^{\omega,b}(x, x')$ is $1-$semi-concave in the second component, and is $(C_m^\omega+1)-$semi-concave in the first component. 
	\item If $(x_j, v_j)_{j=m}^n$  is a minimizer, then for any $k$ such that $m<k<n$, the derivatives $\partial_{2}A^{\omega,b}_{m,k}(x_m, x_k)$ and $\partial_1A^{\omega,b}_{k,n}(x_k, x_n)$ exist. Furthermore, 
	$$\partial_{2}A^{\omega,b}_{m,k}(x_m, x_k) = -\partial_1A^{\omega,b}_{k,n}(x_k, x_n) = v_k -b.$$ 
	\item If $(x_j, v_j)_{j=m}^n$  is a minimizer, then $-v_m + b$ is a subdifferential of $A_{m,n}^{\omega, b}(\cdot, x_n)$ and $v_n + b$ is a subdifferential of $A_{m,n}^{\omega,b}(x_m, \cdot)$.
\end{itemize}
\end{lemma}
\begin{proof}
	The first two conclusions follow directly from the definition. For the third statement, note that  a function of $x$,  
	$$A^{\omega,b}_{m, m+1}(x,x') = \inf_{\tilde{x}'=x' \mod  \Z^d}\left\{\frac12(\tilde{x}'-x)^2 - b\cdot (\tilde{x}'-x) - F_m^\omega(x)\right \}.$$
	For each different lift $\tilde{x}'$, $\frac12(\tilde{x}'-x)^2 - b\cdot (\tilde{x}'-x) - F_m^\omega(x)$ is $C^2$ with a $C^2$ bound $1+C_m^\omega$, and hence are $1+ C_m^\omega$ semi-concave. It follows directly from the definition that minimum of $C-$semi-concave functions are still $C-$semi-concave. Similarly, we conclude that $A_{m, m+1}^{\omega,b}$ is $1-$semi-concave in the second component. 

	To prove the semi-concavity in general, by the first statement, 
	\begin{equation}\label{eq:comparison}
	A^{\omega,b}_{m,n}(x_m, x_n)=\min_{x_{m+1}\in \T^d}\{A_{m,n}^{\omega,b}(x_m,x_{m+1})+ A_{m+1,n}^{\omega,b}(x_{m+1}, x_n) \},
\end{equation}  
is $(1+C_m^\omega)-$semi-concave in the first component. Similarly, the semi-concavity of  $A_{n-1, n}$ in the second component implies the same semi-concavity for $A_{m,n}$. 

Note that if $(x_m, v_m)$, $(x_{m+1},v_{m+1})$ is a minimizer of $A^{\omega,b}_{m, m+1}(x_m,x_{m+1})$, we have $\partial_2A_{m, m+1}^{\omega,b}=v_{m+1}-b$ and $-\partial_1A_{m,m+1}^{\omega, b}=v_m-b$. This implies the fourth conclusion. 

Assume that $x_{m+1}$ achieves the minimum in (\ref{eq:comparison}). Using the definition of subdifferential, we have the subdifferential of $A_{m,m+1}(x_m,x_{m+1})$ in the first component is also a subdifferential of $A_{m,n}(x_m, x_n)$ in the first component. The fifth conclusion follows.  
\end{proof}

In particular, we have the following corollary of the third conclusion of Lemma~\ref{properties}. 
\begin{corollary}
	For any $\varphi\in C(\T^d)$, the function $K_{m,n}^{\omega,b}\varphi(x)$ is $1-$semi-concave; the function $-\tilde{K}_{m,n}^{\omega,b}\varphi (x)$ is $C_m^\omega-$semi-concave. 
\end{corollary}

By Theorem~\ref{backward-min}, we obtain the following. 

\begin{corollary}
	For $n\le n_0$,  $\psi^-(\cdot,n)$ is $1-$semi-concave; for $m\ge n_0$, $-\psi^+(\cdot, m)$ is $C_m^\omega-$semi-concave. 
\end{corollary}

Since the functions $\psi^\pm$ are semi-concave, by Lemma~\ref{Lipshitz} they are Lipshitz. It follows from the Radmacher theorem that they are Lebesgue almost everywhere differentiable. The points of differentiability are tied to the minimizers. The following statement is proved in \cite{IK03}. 

\begin{proposition}[\cite{IK03}]\label{uniqueness}
The full orbit of $(x_0, v_0)$, where $v_0=\nabla \psi^-(x_0,0)+b$, is a global minimizer if and only if $x_0$ is a point of minimum for the function $\psi^-(x,0)-\psi^+(x,0)$. 
\end{proposition}

To prove  the uniqueness of the global minimizer, we only need to show that the function $\psi^-(x,0)-\psi^+(x,0)$ has a unique minimum on $\T^d$. Let us consider potentials of the type (\ref{eq:F-omega}). Given a family of potentials $F_j^\omega=\sum\xi^i_j(\omega)F_i(x)$, let $c=\{\xi_0^1, \cdots, \xi_0^M\}$. We treat $c$ as a parameter of the system. In the following lemma, we will show that the function $\psi^-(x,0)-\psi^+(x,0)$ decompose into a semi-concave part independent of $c$ and a smooth function depending on $c$ and $x$. 

Denote $\rho_b(x,x')=\inf_{m\in \Z^d}\|x'-x-b+m\|$ and 
\[
	\psi(x)=\psi^-(x,0)+\inf_{x_1\in \T^d}\{- \psi^+(x_1,1)  + \frac12 \rho_b(x,x_1)\}, 
\]
we have:
\begin{lemma}\label{var-prob}
$\psi$ is semi-concave and 
$$\psi^-(x,0)-\psi^+(x,0)    = \psi(x) -\sum_{i=1}^M c_iF_i(x).$$
\end{lemma}
\begin{proof}
	The functions $A_{m,0}$ for $m<0$ and the functions $A_{k, n}$ for $1\le k<n$ are independent of $c$. As a consequence the functions $\psi^-(x,m)$ for $m\ge 0$ and the functions $\psi^+(x,k)$ for $k\ge 1$ are all independent of $c$. 

	On the other hand, we have
	\begin{multline*}
		\psi^+(x,0)  = K_{0,1}^+\psi^+(x,1) = \sup_{x_1\in \T^d}\{\psi^+(x_1, 1)-A_{0,1}^{\omega,b}(x, x_1)\}\\
		= \sup_{x_1\in \T^d} \{ \psi^+(x_1,1) - \frac12\rho_b(x, x_1) + \sum_{i=1}^M c_i F_i(x) \}, 
	\end{multline*}
	we obtain the formula in our lemma.  $\psi$ is a semi-concave function since it's the minimum of a family of uniformly semi-concave functions. 
\end{proof}

Proposition~\ref{uniqueness} and Lemma~\ref{var-prob} reduces the uniqueness of the minimum for $\psi^-(x,0)-\psi^+(x,0)$ to the uniqueness of the minimum for $\psi(x)+\sum_{i=1}^M c_i F_i(x)$. The following general statement about the minimum of variational problem implies Theorem~\ref{global-min}. 

\begin{lemma}\label{generic}\cite{IK03}
Let $V(x,c)$ be a $C^2$ function and $\psi(x)$ a continuous function. Assume in addition that, for each $c\in \R^M$, $\frac{\partial V(\cdot, c)}{\partial c}:\T^d \to \R^M$ is one-to-one. Then for Lebesgue a.e. $c\in \R^M$, the function
$$ H(x,c)=\psi(x) + V(x,c)$$ 
has a unique minimum as a function of $x$. 
\end{lemma}

We will prove a series of progressively stronger statements about the variational problem 
$$  \inf_{x\in \T^d} H(x,c)= \inf_{x\in \T^d}\{\psi(x) + V(x,c)\}, $$
in the form of Propositions \ref{nondegenerate}, \ref{integrability} and \ref{strict-nondeg}. These finer properties of the variational problem lead to finer properties of the global minimizer. In particular, the proof of Main Theorem~\ref{hyperbolic} uses Propositions \ref{nondegenerate} and \ref{integrability}, and the proof of Main Theorem~\ref{smooth-vis} follows from Proposition~\ref{strict-nondeg}. Note that below we are using the stronger assumption that $\frac{\partial V(\cdot, c)}{\partial c}$ is an embedding. 

The first statement says that the unique minimum of the variational problem is also a \emph{nondegenerate} minimum. To define nondegeneracy properly, we invoke some definitions from non-smooth analysis. Assume that $f:\T^d\to \R$ is a semi-concave function and that $f'(x_0)$ exists. We define the \emph{second subderivative} of $f$ (See \cite{RW98} for more background) at $x_0$ to be a function $d^2f(x_0):\R^d \to \R\cup\{-\infty\}\cup \{\infty\}$ given by 
\begin{equation}
	\label{eq:second-sub}
	d^2f(x_0)(w)=\liminf_{\tau\to 0+}\frac{f(x_0+ \tau w)-f(x_0)-f'(x_0)(w)}{\frac12\tau^2}. 
\end{equation}
For semi-concave functions, the second subderivative is bounded from above, but $-\infty$ is possible. Throughout the paper, $d^2$ always denotes the second subderivative as defined above,  while the notation $D$ will be used for regular derivatives, and $\partial$ for regular partial derivatives. 

\begin{proposition}\label{nondegenerate}
Let $V(x,c)$ be a $C^2$ function and $\psi(x)$ a semi-concave function. Assume that $\frac{\partial V(\cdot, c)}{\partial c}:\T^d \to \R^M$ is an embedding. Then for 
$$ H(x,c)=\psi(x) + V(x,c),$$ 
and Lebesgue a.e. $c\in \R^M$, the unique minimum  $x(c)$ of $H(x,c)$ is nondegenerate in the sense that, there exists a positive Borel measurable function $a:\R^M \to \R$ with
\begin{equation}
	\label{eq:ac}
	d^2_{x}H(x(c),c)(v)\ge a(c)|v|^2, \quad v\in \R^d.  
\end{equation}
\end{proposition}
While the second subderivative is only defined at points of differentiability, we will see later that the function $H(x,c)$ is differentiable in $x$ at every point of its minimum. 

We also need some quantitative estimates of the function $a(c)$ in Proposition~\ref{nondegenerate}.
\begin{proposition}\label{integrability}
Assume that $V(x,c)=-\sum_{i=1}^Mc_i F_i(x)$ and that the map (\ref{eq:Fi}) is an embedding. Then Proposition~\ref{nondegenerate} applies. Furthermore, if $\rho$ is a density satisfying assumption 5,  then there exists a constant $A(F)$ depending only on $F_1, \cdots, F_M$ such that the function $a(c)$ in (\ref{eq:ac}) satisfies 
$$ \int a(c)^{-1} \rho(c) dc \le A(F). $$
\end{proposition}

The next proposition proves a stronger sense of nondegeneracy. While (\ref{eq:ac}) implies that on a small neighbourhood of $x(c)$, the function $H(x,c)$ is bounded from below by a quadratic function, Proposition~\ref{strict-nondeg} states that the size of this neighbourhood is uniform in $c$. The only cost is a small loss to the power in the integrability condition. 

\begin{proposition}\label{strict-nondeg}
Assume that $V(x,c)=-\sum_{i=1}^Mc_i F_i(x)$ and that the map (\ref{eq:Fi}) is an embedding. There exists a constant $\barA(F)$ depending only on $(F_1, \cdots, F_M)$ and a positive Borel measurable  function $\bara :\R^M\to \R$ with
$$\int \bara(c)^{-\frac12}\rho(c)dc < \barA(F), $$
and a  constant $r(F)>0$ depending only on $(F_1,\cdots, F_M)$, such that 
$$ H(x',c) - H(x(c),c) \ge \bara(c) |x'-x(c)|^2, \quad |x-x(c)|\le r(F). $$
\end{proposition}

Note that it is in fact possible to prove all our theorems using Proposition~\ref{strict-nondeg} alone. We still state Propositions~\ref{nondegenerate} and \ref{integrability} to stress the fact that the stronger form of nondegeneracy is only needed for the proof of Main Theorem~\ref{smooth-vis}. 

Propositions \ref{nondegenerate}, \ref{integrability} and \ref{strict-nondeg} are proved using variational analytic methods, and will be deferred to the end of the paper (Sections \ref{sec:gen-nondeg}, \ref{sec:alexandrov} and \ref{sec:uniform}).

In Sections \ref{sec:green-bundles} and \ref{sec:nonzero-exp}, we prove Main Theorem~\ref{hyperbolic} assuming Propositions \ref{nondegenerate} and \ref{integrability}. In sections \ref{sec:smoothness} and \ref{sec:gen-nondeg}, we prove Main Theorem~\ref{smooth-vis} assuming Proposition~\ref{strict-nondeg}. 

\section{The Green bundle and the nondegeneracy of the minimizer}
\label{sec:green-bundles}

The nondegeneracy of the minimum from Proposition~\ref{nondegenerate} is connected with the Lyapunov exponents, via the so-called Green bundles (see \cite{Gre58},\cite{BM04},\cite{Arn10}). Roughly speaking, Proposition~\ref{nondegenerate} implies the transversality of the Green bundles, and the transversality of the Green bundles implies nonzero Lyapunov exponents. We will prove the first implication in this section, and the second implication in Section~\ref{sec:nonzero-exp}. 

Let $(\delta x, \delta v)$ be the coordinates of the tangent space adapted to the coordinates $(x,v)$.  At each $(x,v)$, we define the \emph{vertical space} $V(x,v)=\{(0,\delta v)\}$ and the \emph{horizontal space} $H(x,v)=\{(\delta x, 0)\}$. An orbit $(x_j, v_j)_{j\in I}$ is called \emph{disconjugate} if for any $m, n\in \Z$, $[m, n]\subsetneq I$, we have that $D\Phi^\omega_{m,n}(x_m, v_m) V(x_m, v_m)\cap V(x_n, v_n)=\{0\}$. It is well known that minimizing orbits have no conjugate points. 

\begin{lemma}(see \cite{GI99}, \cite{Arnaud2013})
If $(x_j,v_j)_{j\in I}$ is a minimizer, then $(x_j,v_j)_{j\in I}$ is disconjugate. 
\end{lemma}

For the rest of this section, we fix a global minimizer $(x_j, v_j)_{j\in \Z}$. For $n\in \Z$ and $k\in \N$, we define a subspace $\check{U}_k^\omega(x_n, v_n)$ of $T_{(x_n,v_n)}(\T^d\times \R^d)$ by
$$ \check{U}^\omega_k(x_n,v_n)=D\Phi_{n-k, n}^\omega V(x_{n-k}, v_{n-k}).$$
Similarly, we may define a subspace $\check{S}^\omega_k(x_n, v_n)$ of $T_{(x_n, v_n)}(\T^d \times \R^d)$ by 
$$ \check{S}^\omega_k(x_n,v_n)=D(\Phi_{n, n+k}^\omega)^{-1} V(x_{n+k}, v_{n+k}).$$
When we don't need to stress the dependence on the random realization $\omega$, we will drop the supscript $\omega$ for $\check{S}_k$ and $\check{U}_k$. We shall use a standard ordering in the space of symmetric matrices. Namely, given two symmetric matrices $A$ and $B$, we say that $A\ge B$ if $A-B$ is positive semi-definite. We say that $A>B$ if $A-B$ is positive definite.

The following statement for the case of a sequence of twist maps is due to Bialy and MacKay (\cite{BM04}). 
\begin{lemma}\cite{BM04}\label{greenbundles}
Assume that $(x_j, v_j)_{j\in \Z}$ is disconjugate. We have the following conclusions. 
\begin{enumerate}
	\item   There exists $d\times d$ symmetric matrices $S_k(x_n, v_n)$ and $U_k(x_n, v_n)$ such that $\check{S}_k(x_n, v_n)=\{\delta v= S_k(x_n, v_n) \delta x\}$ and $\check{U}_k(x_n, v_n)=\{\delta v= U_k(x_n, v_n) \delta x\}$. 
	\item  We have
	$$ U_1(x_n, v_n)> \cdots  U_{k}(x_n, v_n)> \cdots > S_{k}(x_n, v_n) > S_1(x_n,v_n). $$
	\item There exists symmetric matrices $\bbS(x_n, v_n)$  and $\bbU(x_n, v_n)$ such that 
	$$\lim_{k\to \infty}U_k(x_n,v_n)=\bbU(x_n, v_n),  \lim_{k\to \infty}S_k(x_n,v_n)=\bbS(x_n, v_n)$$
\end{enumerate}
\end{lemma}

Let $\check{S}(x_n, v_n)=\{\delta v= \bbS(x_n, v_n) \delta x\}$ and $\check{U}(x_n, v_n)=\{\delta v= \bbU(x_n, v_n) \delta x\}$. These subspaces are traditionally called the negative and positive Green bundles. In this paper, we will use the name \emph{stable and unstable Green bundles} to avoid possible confusions with the positive and negative viscosity solutions. It follows from Lemma~\ref{greenbundles} that the bundles $\check{S}(x_n, v_n)$ and $\check{U}(s_n, v_n)$ are invariant in the sense that 
$$ D\Phi^\omega_{m,n}\check{S}(x_m, v_m)=\check{S}(x_n, v_n), \quad D\Phi^\omega_{m,n}\check{U}(x_m, v_m)=\check{U}(x_n, v_n). $$

For a positive semi-definite matrix $A$, let $m(A)$ denote its smallest eigenvalue. It's easy to see that  $m(A)=\|A^{-1}\|^{-1}$.  It follows from Lemma~\ref{greenbundles} that $\bbU(x_n, v_n)\ge \bbS(x_n, v_n)$.  The subspaces $\check{S}$ and $\check{U}$ are transversal if and only if
$$ m(\bbU(x_n, v_n)-\bbS(x_n,v_n))>0.$$

The transversality of the Green bundles at the global minimizer is related to the viscosity solutions $\psi^\pm$.  We will assume that the parameters $\omega$ and $b$ are chosen such that the viscosity solutions $\psi^\pm$ exist and are unique. We now state the main conclusion of this section. 

\begin{proposition}\label{transversal-bundle}
Assume that the function $\psi^-(\cdot, 0)-\psi^+(\cdot, 0)$ has a unique minimum at $x_0$, and that there exists $a>0$ such that 
\begin{equation}
	\label{eq:2nd-der} 
	d^2(\psi^--\psi^+)(x_0,0)(v)\ge  a\|v\|^2, \quad v\in \R^d, 
\end{equation}
where $d^2$ stands for the second subderivative in $x$, see \eqref{eq:second-sub}. 
Let $(x_j,v_j)_{j\in \Z}$ be the global minimizer corresponding to the minimum $x_0$. We have 
$$ m(\bbU(x_0,v_0)-\bbS(x_0,v_0))\ge a. $$
\end{proposition}

The proof of Proposition~\ref{transversal-bundle} is split into several lemmas. The proofs of  Lemma~\ref{2nd-der-action} and formula (\ref{eq:2nd-der-comp}) follows the ideas in section 4 of \cite{Arn12}. We provide complete proofs for the convenience of the reader. 

\begin{lemma}\label{2nd-der-action}
Let $\{x_j, v_j\}_{j\in \Z}$ be a global minimizer. Then for any $m<n$ in $\Z$, the function $A_{m,n}(x,x')$ is $C^2$ in a neighbourhood of $(x_m, x_n)$. Furthermore, 
$$ U_{n-m}(x_n,v_n)=\partial^2_{22}A_{m,n}(x_m,x_n), \quad  S_{n-m}(x_m,v_m)=- \partial^2_{11}A_{m,n}(x_m, x_n).$$
\end{lemma}

The subdifferential of a semi-concave function is upper semi-continuous as a set function. In particular, if the subdifferential is unique at $x_0$, then any subdifferential at $x_n\to x_0$ must converge to the derivative at $x_0$.

\begin{lemma}[\cite{RW98}, Proposition 8.7]\label{semi-cont}
Let $f(x)$ be a semi-concave function and $x_0$ be such that $f'(x_0)$ exists. Then for any $x_n\to x$ and  $l_{x_n}$  any subdifferential of $f(x)$ at $x_n$,  we have 
$$ l_{x_n}\to f'(x_0). $$
\end{lemma}

\begin{proof}[Proof of Lemma~\ref{2nd-der-action}]
	We have that $\Phi_{m,n}$ is a $C^1$ map in a neighbourhood of $(x_m, v_m)$ and that $\Phi_{m,n}(x_m, v_m)=(x_n,v_n)$. Since $(x_j,v_j)$ is a disconjugate orbit, $D\Phi_{m,n}V(x_m, v_m)\cap V(x_n,v_n)=\{0\}$. Denote 
	$$ 
	\Phi_{m,n}(x,v)=(x',v'), \quad
	D\Phi_{m,n}(x,v)=
	\begin{bmatrix}
		\bbA_{m,n} & \bbB_{m,n} \\
		\bbC_{m,n} & \bbD_{m,n}
	\end{bmatrix}(x,v).
	$$
	We have $\det \frac{\partial x'}{\partial v}(x_m, v_m) = \det \bbB_{m,n}(x_m,v_m)\ne 0$.  By the implicit function theorem, there exists unique $C^1$ functions $v=v(x,x')$ and $v'=v'(x,x')$ such that $\Phi_{m,n}(x,v(x,x'))=(x',v'(x,x'))$ in a neighbourhood of $\{(x_m, v_m)\}\times\{(x_n,v_n)\}$. 

	Let $(x,x')$ be sufficiently close to $(x_m,x_n)$. Let $(y_i,w_i)_{j=m}^n$ be any minimizing orbit for $A_{m,n}(x,x')$ and denote $v=w_m$, $v'=w_n$. By Lemma~\ref{properties}, $A_{m,n}$ is differentiable at $(x_m,x_n)$, $v'\to v_n$ and $v\to v_m$ as $(x, x')\to (x_m,x_n)$. Assume that  $(x,v,x', v')$ is so close to $(x_m, v_m, x_n,v_n)$  that the implicit function theorem applies. Then we have that $v=v(x,x')$ and $v'=v'(x,x')$ are well defined $C^1$ functions of $x$ and $x'$. Since $v=-\partial_1A_{m,n}(x,x')+b$ and $v'=\partial_2A_{m,n}(x,x')+b$, we conclude that $A_{m,n}(x,x')$ is a $C^2$ function. 

	Viewing $x',v'$ as functions of $(x,v)$, we have that 
	$$ \partial_2A_{m,n}(x, x'(x,v))=v'(x,v)-b. $$
	Differentiating both sides with respect to $v$, we have
	$$ \partial^2_{22}A_{m,n}(x_m,x_n)=\left(\frac{\partial v'}{\partial v}\right)\left(\frac{\partial x'}{\partial v}\right)^{-1}(x_m,v_m)=(\bbD_{m,n})(\bbB_{m,n})^{-1}(x_m,v_m).$$
	Using the definition of $U_{m-n}$, we have $(\bbD_{m,n})(\bbB_{m,n})^{-1}(x_m,v_m)=U_{n-m}(x_m,v_m)$. The conclusion about $S_{n-m}(x_m,v_m)$can be proved similarly using the map $\Phi_{m,n}^{-1}$. 
\end{proof}

\begin{lemma}\label{comparison}
Let $(x_j, v_j)_{j\in \Z}$ be a global minimizer. Then for $m < n \in\N$ and $w\in \R^d$, we have
$$ \langle \partial^2_{22}A_{m,n}(x_m,x_n)w,w\rangle \ge d^2\psi^-(x_n,n)(w),$$ and 
$$   \langle \partial^2_{11}A_{m, n}(x_m,x_n)w,w\rangle \ge d^2(-\psi^+(x_m,m))(w). $$
\end{lemma}
\begin{proof}
	Since $K_{m,n}\psi^-(x,m)=\psi^-(x,n)$, we have
	$$ \psi^-(x_n, n)=\inf_{x'\in \T^d}\{\psi^-(x',m)+ A_{m,n}(x', x_n)\}=\psi^-(x_m,m)+A_{m,n}(x_m, x_n).$$
	For any other $x\in \T^d$, we have
	$$ \psi^-(x,0)=\inf_{x'\in \T^d}\{\psi^-(x',m)+ A_{m,n}(x',x)\}\le \psi^-(x_m,m)+A_{m,n}(x_m,x).$$
	Furthermore, 
	$$ \nabla\psi^-(x_n, n)=\partial_2A_{m,n}(x_m, x_0)=v_0-b. $$
	Combining the last three formulas, we have
	\begin{multline}\label{eq:2nd-der-comp}
	\psi^-(x,0)-\psi^-(x_0,0)-\langle \nabla \psi^-(x_0,0), x-x_0\rangle \\
	\le A_{m,n}(x_m, x) - A_{-k}(x_m, x_0) - \langle \partial_2 A_{m,n}(x_m, x_0), x-x_0\rangle. 
\end{multline}
Take $x-x_0 = \tau w$, divide by $\tau^2$ and take lower limit as $\tau\to 0+$, we conclude
$$ d^2\psi^-(x_0, 0)(w) \le \langle \partial^2_{22}A(x_k, x_0)w, w\rangle.$$ 
The first inequality of the lemma follows. 

Similar estimates with $\varphi^-$ replaced with $-\varphi^+$ prove the second inequality. 
\end{proof}

\begin{proof}
	[Proof of Proposition~\ref{transversal-bundle}]
	Applying Lemma~\ref{2nd-der-action} and \ref{comparison}, we have 
	\[
		\begin{aligned}
			& 	m(U_{-k}(x_0, v_0) - S_{-k}(x_0, v_0)) \ge m\left(  \partial^2_{22} A_{-k, 0}(x_{-k}, x_0) + \partial^2_{11} A_{-k, 0}(x_{-k}, x_0) \right)  \\
			& \ge \inf_{\|w\| = 1} d^2(\psi^- - \psi^+)(x_0, 0)(w) \ge a, 
		\end{aligned}
	\]
	the proposition follows by taking $k \to \infty$. 
\end{proof}

\section{Nonzero Lyapunov exponents}
\label{sec:nonzero-exp}

Recall that the family of maps $\Phi^\omega_j$ may be viewed as a single transformation $\hat\Phi$ acting on $\T^d\times\R^d\times \Omega$ given by \eqref{eq:random-map}. Define $\bbS, \bbU: \Omega\to M(d\times d)$ by $\bbS(\omega)=\bbS(x_0^\omega, v_0^\omega, \omega)$ and $\bbU(\omega)=\bbU(x_0^\omega, v_0^\omega, \omega)$, where $(x^\omega_j,v^\omega_j)_{j\in \Z}$ is the unique global minimizer guaranteed by Theorem~\ref{global-min}. We have 
$$\bbS(x^\omega_n, v^\omega_n)=\bbS(\theta^n \omega), \quad \bbU(x^\omega_n, v^\omega_n)=\bbU(\theta^n \omega).$$

With the assumptions of Main Theorem~\ref{hyperbolic}, the assumptions of Propositions \ref{nondegenerate} and \ref{integrability} are satisfied. By combining Proposition~\ref{nondegenerate} and Proposition~\ref{transversal-bundle}, we obtain that there exists a positive measurable function $a(\omega)$ such that 
\begin{equation}
	\label{eq:trans} 
	m(\bbU(\omega)-\bbS(\omega))\ge a(\omega).
\end{equation}
Moreover, using Proposition~\ref{integrability}, we have
\begin{equation}
	\label{int-est} 
	\int a(\omega)^{-1} dP(\omega) \le A(F) <\infty.
\end{equation}
It suffices to show these estimates imply Main Theorem~\ref{hyperbolic}. 

Before discussing the Lyapunov exponents of the cocycle $D\Phi^\omega_0$, we need to show that they are well defined.  We first describe some properties of the symplectic map $D\Phi_j^\omega$.We have
$$ D\Phi_j^\omega=
\begin{bmatrix}
	\bbA_j^\omega & \bbB_j^\omega \\
	\bbC_j^\omega & \bbD_j^\omega
\end{bmatrix},   $$
where 
\begin{equation}
	\label{eq:formula}
	\bbA_j^\omega=I + \partial^2 F_j^\omega,\quad \bbB^\omega_j=\bbD_j^\omega=I, \quad \bbC_j^\omega=-\partial^2F^\omega.
\end{equation}
Although we have an explicit formula for the matrix, the discussions that follow only use $\det \bbB_j^\omega\ne 0$ and some norm estimates. 

Any symplectic matrix given in the block form $[\bbA, \bbB; \bbC, \bbD]$ has the following properties. 
\begin{itemize}
	\item $\bbA^T\bbC=\bbC^T\bbA$, $\bbB^T\bbD=\bbD^T\bbB$, $\bbA^T\bbD-\bbC^T\bbB=I$. 
	\item $\bbA\bbB^T=\bbB\bbA^T$, $\bbC\bbD^T=\bbD\bbC^T$, $\bbA\bbD^T-\bbB\bbC^T=I$. 
	\item 
	$$
	\begin{bmatrix}
		\bbA & \bbB \\
		\bbC & \bbD 
	\end{bmatrix}^{-1}=
	\begin{bmatrix}
		\bbD^T & -\bbB^T \\
		-\bbC^T & \bbA^T
	\end{bmatrix}.
	$$
\end{itemize}
From the explicit formula (\ref{eq:formula}) for $D\Phi^\omega_j$ and its inverse, we have that $\|D\Phi^\omega_j\|, \|D(\Phi^\omega_j)^{-1}\| \le 1+ C_j^\omega$, where $C_j^\omega=\|\partial^2 F_j^\omega\|$. 

Since $\E(1+C_0^\omega)<\infty$, we have $\E(\log(1+C_0^\omega))<\infty$. It follows that 
$$ \log^+\|D\Phi^\omega_j\|, \log^+\|D(\Phi^\omega_j)^{-1}\| \in L^1(\nu), $$
where $\log^+(x)=\max\{\log x, 0\}$. By Oseledets' theorem for cocycles (see \cite{BP07}, Theorem 3.4.3), the Lyapunov exponents $\lambda_i(\nu)$ are well defined. We say that a positive function $g: \Omega \to \R$ is tempered if almost surely, 
$$ \lim_{n\to \pm \infty} \frac{1}{n} \log g(\theta^n \omega)=0.$$
We have the following lemma:
\begin{lemma}[\cite{BP07}, Lemma 2.1.5]\label{tempered}
If $\log^+g(\omega), \log^+(g(\omega)^{-1})\in L^1(d\omega)$, then $g$ is tempered. 
\end{lemma}

Our Main Theorem~\ref{hyperbolic} follows immediately from Proposition~\ref{nonzero-exp} below. The proof mostly follow \cite{Arnaud2013} with the additional ingredient (\ref{int-est}). 
\begin{proposition}\label{nonzero-exp}
Assume that the Green bundles defined along the global minimizer $(x_j,v_j)$ satisfies (\ref{eq:trans}) and (\ref{int-est}).  Then we have
$$ \chi_{d+1}(\nu)\ge \frac12 \int \log \left( 1+ \frac{1}{2+C_0^\omega} m(\bbU(\omega)-\bbS(\omega))  \right)dP(\omega)>0.$$
\end{proposition}

\begin{proof}
	To simplify the notations, we omit the supscript $\omega$. The invariance of  the bundles $\bbS(x_n, v_n)$ and $\bbU(x_n, v_n)$ corresponds to the following statement: Given a vector $h\in \R^d$, if $D\Phi_j(h, \bbS(x_j, v_j))=(h',w')$, then $w'=\bbS(x_{j+1}, v_{j+1})h'$. To further simplify notations, let us write $z_j=(x_j, v_j)$, $\bbS_j=\bbS(z_j)$ and $\bbU_j=\bbU_j(z_j)$. Expressing the above invariance relation in the matrix form, we have: 
	\begin{equation}
		\label{eq:s-inv} 
		\bbS_{j+1}(\bbA_j+ \bbB_j\bbS_j)=\bbC_j + \bbD_j\bbS_j.
	\end{equation}
	By the same reasoning, we have:
	\begin{equation}
		\label{eq:u-inv} 
		\bbU_{j+1}(\bbA_j+ \bbB_j\bbU_j)=\bbC_j + \bbD_j\bbU_j.
	\end{equation}

	We would like to understand the matrix product $D\Phi_{m,n}(x_m)=\prod_{j=m}^{n-1}D\Phi_j(x_j)$ by introducing a change of coordinates. Let 
	$$ Q=
	\begin{bmatrix}
		I & I \\
		\bbS & \bbU 
	\end{bmatrix}
	\begin{bmatrix}
		(\bbU-\bbS)^{-\frac12} &  0\\
		0 & (\bbU-\bbS)^{-\frac12}
	\end{bmatrix}.
	$$
	By a direct calculation, we have that $Q$ is a symplectic matrix, and 
	$$ Q^{-1}=
	\begin{bmatrix}
		(\bbU-\bbS)^{-\frac12} &  0\\
		0 & (\bbU-\bbS)^{-\frac12}
	\end{bmatrix}
	\begin{bmatrix}
		\bbU & -I \\
		-\bbS & I 
	\end{bmatrix}.
	$$
	Note that 
	\begin{multline*}
		\begin{bmatrix}
			\bbU_{j+1} & -I \\
			-\bbS_{j+1} & I 
		\end{bmatrix}
		\begin{bmatrix}
			\bbA_j & \bbB_j \\
			\bbC_j & \bbD_j
		\end{bmatrix}
		\begin{bmatrix}
			I & I \\
			\bbS_j & \bbU_j 
		\end{bmatrix} =  \\
		\begin{bmatrix}
			-\bbS_{j+1}(\bbA_j + \bbB_j\bbU_j) + (\bbC_j + \bbD_j \bbU_j)  & 0 \\
			0 &  \bbU_{j+1}(\bbA_j + \bbB_j\bbS_j) + (\bbC_j + \bbD_j \bbS_j)
		\end{bmatrix} \\
		= \begin{bmatrix}
		(\bbU_{j+1}-\bbS_{j+1})(\bbA_j + \bbB_j\bbU_j)  & 0 \\
		0 &  (\bbU_{j+1}-\bbS_{j+1})(\bbA_j + \bbB_j\bbS_j)
	\end{bmatrix}  
\end{multline*}
The last line of the above calculation is due to (\ref{eq:s-inv}) and (\ref{eq:u-inv}). We obtain
$$ Q(z_{j+1})^{-1} D\Phi_j(z_j)\, Q(z_j): =  \begin{bmatrix}
M_j & 0 \\
0 & N_j
\end{bmatrix}, $$
where
$$ M_j = (\bbU_{j+1}-\bbS_{j+1})^{\frac12} (\bbA_j + \bbB_j \bbS_j) (\bbU_j-\bbS_j)^{-\frac12}, $$
$$ N_j = (\bbU_{j+1}-\bbS_{j+1})^{\frac12} (\bbA_j + \bbB_j \bbU_j) (\bbU_j-\bbS_j)^{-\frac12}. $$
Since the matrix is symplectic, we have  $(N_j)^TM_j=I$. We have the following computation:
\begin{align*}
	N_j^TN_j & = M_j^{-1}N_j \\
	& =  (\bbU_j-\bbS_j)^{\frac12} (\bbA_j + \bbB_j \bbS_j)^{-1} (\bbA_j + \bbB_j\bbU_j) (\bbU_j-\bbS_j)^{-\frac12}\\
	& =  I + (\bbU_j-\bbS_j)^{\frac12} (\bbB_j^{-1}\bbA_j + \bbS_j)^{-1} (\bbU_j-\bbS_j) (\bbU_j-\bbS_j)^{-\frac12} \\
	& =  I + (\bbU_j-\bbS_j)^{\frac12} (\bbB_j^{-1}\bbA_j + \bbS_j)^{-1} (\bbU_j-\bbS_j)^{\frac12}
\end{align*}
We claim that the matrix $\bbB_j^{-1}\bbA_j + \bbS_j$ is positive definite and that the conorm
\begin{equation}
	\label{eq:conorm}
	m((B_j^{-1}A_j + \bbS_j)^{-1})\ge \frac{1}{2+ C_j^\omega}.
\end{equation}
To see this, recall that the matrix $S_1(z_j)$  is defined by $D(\Phi_j)^{-1}V(z_{j+1})$. Take a vertical vector $(0, \delta v)\in V(z_j)$, we have that 
\begin{equation}
	\label{eq:S_1}  
	D(\Phi_j)^{-1}
	\begin{bmatrix}
		0 \\ \delta v
	\end{bmatrix} = 
	\begin{bmatrix}
		\bbD_j^T & -\bbB_j^T \\
		-\bbC_j^T & \bbA_j^T 
	\end{bmatrix}
	\begin{bmatrix}
		0 \\ \delta v
	\end{bmatrix} =
	\begin{bmatrix}
		-\bbB_j^T \delta v \\ \bbA_j^T \delta v
	\end{bmatrix}.
\end{equation}
It follows that $S_1(z_j)=-(\bbB_j^{-1}\bbA_j)^T = -\bbB_j^{-1}\bbA_j$. Since $\bbS_j>S_1(z_j)$, we have that $\bbB_j^{-1}\bbA_j + \bbS_j= \bbS_j- S_1(z_j)>0$. Since $U_1(z_j)>\bbU_j>\bbS_j$, we have  
$$m((\bbB_j^{-1}\bbA_j + \bbS_j)^{-1}) = \|\bbB_j^{-1}\bbA_j + \bbS_j\|^{-1} \ge \|U_1(z_j)-S_1(z_j)\|^{-1}.$$
By a calculation similar to the one in \eqref{eq:S_1}, we have  $U_1(z_j)=\bbD_{j-1} \bbB_{j-1}^{-1}$.  From the explicit formula (\ref{eq:formula}) of $D\Phi_j$, it is easy to see that $\|U_1(z_j)-S_1(z_j)\|\le 2+ C_j^\omega$. The estimate \eqref{eq:conorm} follows. 

Using the estimates obtained, we have that 
\begin{equation}
	\label{eq:M-norm}
	m(N_j^TN_j)\ge 1 + \frac{1}{2+C_j^\omega}m((\bbU_j-\bbS_j)).
\end{equation}

We are now ready to estimate the Lyapunov exponent $\lambda_{d+1}(\nu)$. Since 
$$ D\Phi_{0,n}=Q(z_{n})\prod_{j=n-1}^0
\begin{bmatrix}
	M_j & 0 \\
	0 & N_j
\end{bmatrix} Q(z_0)^{-1},
$$
we have
$$ \left(Q(z_{n})^{-1}D\Phi_{0,n}\right)^T\left(Q(z_{n})^{-1}D\Phi_{0,n}\right) = (Q(z_0)^{-1})^T
\begin{bmatrix}
	M_{n,0}^TM_{n,0} & 0 \\
	0 & N_{n,0}^TN_{n,0}
\end{bmatrix}
Q(z_0)^{-1},
$$
where $M_{n,0}=\prod_{j=n-1}^0M_j$ and $N_{n,0}=\prod_{j=n-1}^0N_j$. We have the following estimates:
$$ m(N_{n,0}^TN_{n,0})=m((M_{n,0}^TM_{n,0})^{-1})\ge \prod_{j=0}^{n-1}\left(1 + \frac{1}{2+C_j^\omega}m(\bbU_j-\bbS_j)\right).$$
Consider a vector $w\in \R^{2d}$ and let $\tilde{w}=(\tilde{w}_1, \tilde{w}_2)=Q(z_0)^{-1}w$. If $\tilde{w}_2\ne 0$, then 
\begin{multline}\label{lyapunov} \liminf_{n\to \infty}\frac1{n}\log\|Q(z_{n})^{-1}D\Phi_{0,n}w\|^2\\
\ge \liminf_{n\to \infty}\frac1{n}\log \prod_{j=0}^{n-1}\left(1 + \frac{1}{2+C_j^\omega}m(\bbU_j-\bbS_j)\right)\|\tilde{w}_1\|^2\\
= \liminf_{n\to \infty}\frac1{n}\sum_{j=0}^{n-1}\log\left(1 + \frac{1}{2+C_j^\omega}m(\bbU_j-\bbS_j)\right).
\end{multline}
If $\tilde{w}_2=0$, then 
$$\limsup_{n\to \infty}\frac1{n}\log\|Q(z_{n-1})^{-1}D\Phi_{0,n}w\|^2\le -\liminf_{n\to \infty}\frac1{n}\sum_{j=0}^{n-1}\log\left(1 + \frac{1}{2+C_j^\omega}m(\bbU_j-\bbS_j)\right).$$
The norm  $\|Q(z_{n-1})\|\le (2+ C_{n-1}^\omega)m(\bbU_{n-1}-\bbS_{n-1})^{-\frac12}$.

We have $\log(1+C_0^\omega)\in L^1(dP(\omega))$ by assumption.  Furthermore, since
$$ a(\omega)\le \bbU(\omega)-\bbS(\omega)\le U_1(z_0^\omega)-S_1(z_0^\omega) \le 1 + C_0^\omega,$$
we know that $\log^+ (\bbU(\omega)-\bbS(\omega)), \log^+(\bbU(\omega)-\bbS(\omega))^{-1}\in L^1(dP(\omega))$. Using Lemma~\ref{tempered}, we have 
\begin{equation}
	\label{eq:temp-1}
	\lim_{n\to \infty}\frac{1}{n}\log(1+ C_{n-1}^\omega)=0, \quad
	\lim_{n\to \infty}\frac1{2n}\left|\log m(\bbU(z^\omega_{n-1})-\bbS(z^\omega_{n-1})\right|=0.
\end{equation}

We now finish the proof of the proposition.  If $\tilde{w}_1\ne 0$, using (\ref{lyapunov}), (\ref{eq:temp-1}) and the Birkhoff ergodic theorem, we have
\begin{multline*}
	\liminf_{n\to \infty}\frac1{n}\log\|D\Phi_{0,n}w\| 
	\ge \int\frac12\log\left(1 + \frac{1}{1+C_0^\omega}m((\bbU-\bbS)(z^\omega_0))\right)dP(\omega)
	\\ - \limsup_{n\to \infty}\frac{1}{n}\left(\log(1+ C_{n-1}^\omega) - \frac12\log m(\bbU(z^\omega_{n-1})-\bbS(z^\omega_{n-1}))\right)
	\\ = \int\frac12\log\left(1 + \frac{1}{1+C_0^\omega}m((\bbU-\bbS)(z^\omega_0))\right)dP(\omega)>0.
\end{multline*}
Note that for $\tilde{w}_1=0$ we have 
$$ \limsup_{n\to \infty}\frac1{n}\log\|D\Phi_{0,n}w\| 
\le -\int\frac12\log\left(1 + \frac{1}{1+C_0^\omega}m((\bbU-\bbS)(z^\omega_0))\right)dP(\omega)<0.$$
\end{proof}

\section{Dynamics near the global minimizer}
\label{sec:loc-smooth}
We prove Main Theorem~\ref{smooth-vis} in the following two sections. By assumption, the potentials $F_j^\omega$ are $C^{2+\alpha}$ for some $0< \alpha\le 1$. We shall abuse notation and denote 
$$ C_j^\omega=\|F_j^\omega\|_{2+\alpha}.$$ 
By Assumption 5,  $\mathbb{E}(C_j^\omega)<\infty$. 

As the global minimizer $(x_j,v_j)$ is hyperbolic, we will apply the theory of nonuniform hyperbolic systems (Pesin's theory) to obtain the local stable manifolds $W^s(x_j, v_j)$ and unstable manifolds $W^u(x_j,v_j)$. We restate the unstable part of Main Theorem~\ref{smooth-vis} as follows. 

\begin{theorem}\label{smooth}
There exists a neighbourhood $V(\omega)$ of $(x_0,v_0)$ such that $\psi^-$ is $C^2-$smooth on $V$, and
$$W_0^-:=\{(x, \nabla_x \psi^-(x,0))\} \cap V= W^u(x_0,v_0)\cap V.$$
\end{theorem}
\begin{remark}
	We can prove the corresponding statement for $W_0^+$, namely
	$$W_0^+:=\{(x, -\nabla_x \psi^+(x,0))\} \cap V\subset W^s(x_0,v_0)\cap V$$
	by reversing time. Main Theorem~\ref{smooth-vis} follows. 
\end{remark}

In this section, we will focus on studying the dynamics of orbits close to the global minimizer. This information will be used to prove Theorem~\ref{smooth} in the next section.
The main goal is to obtain precise estimates of the expansion/contraction for the random maps $\Phi_j^\omega$. The methods used here are well known to experts in the field; we will provide proofs for some and refer to \cite{BP07} for others. 

In the proof of Proposition~\ref{nonzero-exp}, we obtained the following reduction of the cocycle $D\Phi_{m,n}(x_j, v_j)$. Let $z_j$ denote $(x_j, v_j)$ and 
$$ Q(z_j)=
\begin{bmatrix}
	I & I \\
	\bbS(z_j) & \bbU(z_j)
\end{bmatrix}
\begin{bmatrix}
	(\bbU-\bbS)(z_j)^{-\frac12} & 0 \\ 
	0 & (\bbU-\bbS)(z_j)^{-\frac12} 
\end{bmatrix}, $$
then $\bbL(z_j):=Q(z_{j+1})^{-1}D\Phi_j Q(z_j) = \diag\{M_j, N_j\}$. Here 
$$ N_j = (\bbU_{j+1}-\bbS_{j+1})^{\frac12} (\bbA_j + \bbB_j \bbU_j) (\bbU_j-\bbS_j)^{-\frac12} $$
and $M_j = (N_j^T)^{-1}$.

We will also use a standard construction in nonuniform hyperbolic theory known as the \emph{tempering kernel}. 
\begin{lemma}[See \cite{BP07}, Lemma 3.5.7.]\label{tempering}
Assume that $g:\Omega\to \R$ is a tempered positive function, that is  almost surely $ \lim_{n\to \pm\infty} \frac{1}{n}\log g(\theta^n\omega) =0$. 
Then for any $\epsilon>0$ there exists positive functions $K^1, K^2:\Omega \to \R$ such that 
almost surely $K^1(\omega) \le g(\omega) \le K^2(\omega)$, and 

$$e^{-\epsilon}\le \frac{K^i(\theta \omega)}{K_i(\omega)} \le e^{\epsilon}, \, i=1,2. $$
\end{lemma}
We define
$$\lambda_j(\omega)=\max\left\{\log \left( 1+ \frac{m(\bbU_j(\omega)-\bbS_j(\omega))}{1+ C_j^\omega}\right), 1\right\}. $$
According to (\ref{eq:M-norm}), we have
$$ \|N_j^{-1}\|^{-1} \ge e^{\lambda_j}, \quad \|M_j\|\le e^{-\lambda_j}. $$
We have $\lambda_j(\omega) = \lambda_0(\theta^j \omega)$, and $1\ge \lambda_0\ge m(\bbU_0(\omega)-\bbS_0(\omega)) \ge a(\omega)$, hence $\lambda_0(\omega)$ is integrable. Denote $\lambda = \E \lambda_0(\omega)$. 
\begin{lemma}\label{slowvarying}
For any $\epsilon>0$, there exists $\kappa(\omega), K_0(\omega)>0, a_0(\omega)>0$ satisfying
$$e^{-\epsilon}\le \frac{K_0(\theta \omega)}{K_0 (\omega)}, \, \frac{a_0(\theta \omega)}{a_0 (\omega)}, \, \frac{\kappa_0(\theta \omega)}{\kappa_0(\omega)}\le e^{\epsilon}$$
such that $2+C_0(\omega)\le K_0(\omega)$, $a_0(\omega)\le m(\bbU_0(\omega)-\bbS_0(\omega))$, and   
$$ \exp{\sum_{k=0}^{n-1} \lambda_{\pm k}(\omega)} \ge \kappa_0(\omega) e^{ n(\lambda -\epsilon)}. $$
\end{lemma}
\begin{proof}
	The existence of $K_0$ and $a_0$ follows from (\ref{eq:temp-1}) and Lemma~\ref{tempering}. For the existence of $\kappa_0$, we define 
	$$ \Gamma_\pm(\omega)=\sup\left\{\frac{e^{\lambda n}e^{-\epsilon n}}{\prod_{k=0}^{n-1}(1+\lambda_0(\theta^{\pm k}\omega))}, \, n\ge 0\right\},$$
	where the term corresponding to $n=0$ is defined to be $1$. 
	Note that 
	$$ \frac{e^{\lambda - \epsilon}}{e^{\lambda_0(\omega)}} \le \frac{\Gamma_\pm(\theta\omega)}{\Gamma_\pm(\omega)} \le  \frac{e^{\lambda_0(\omega)}}{e^{\lambda - \epsilon}},$$
	we have that $\log \frac{\Gamma_\pm(\theta\omega)}{\Gamma_\pm(\omega)} \in L^1(dP(\omega))$. Applying the Birkhoff ergodic theorem to $\log \frac{\Gamma_\pm(\theta\omega)}{\Gamma_\pm(\omega)}$, we obtain that $\Gamma_{\pm}(\omega)$ is tempered. 

	Applying Lemma~\ref{tempering} to $\frac{1}{\Gamma_{\pm}(\omega)}$, we obtain positive functions $\kappa_0^\pm(\omega)\le \frac{1}{\Gamma_{\pm}(\omega)}$ with 
	$$e^{-\epsilon} \le (\kappa^\pm_0(\theta \omega))/(\kappa_0^\pm (\omega)) \le e^{\epsilon}. $$
	Furthermore, 
	$$ \exp{\sum_{k=0}^{n-1}\lambda_{\pm k}}  \ge \frac{1}{\Gamma_\pm(\omega)} e^{\lambda n} e^{-\epsilon n} \ge  \kappa^\pm_0(\omega) e^{n(\lambda-\epsilon)}. $$
	The lemma follows by taking $\kappa_0  = \min\{\kappa_0^+, \kappa_0^-\}$.
\end{proof}
Going forward, we will choose the functions $a_0$, $K_0$, $\kappa_0$ using Lemma~\ref{slowvarying}. However, the parameter $\epsilon$ with which we apply the lemma will be decided later. We will also denote $a_j=a_0 \circ \theta^j$, $K_j= K_0\circ \theta^j$, $\kappa_j=\kappa_0\circ \theta^j$.  

Let $P_j$ be a map from $\R^d\times \R^d$ to a neighbourhood of $(x_j, v_j)$ defined by
\begin{equation}
	\label{eq:Pj}
	P_j(s,u) = Q(x_j, v_j)(s,u)+(x_j,v_j). 
\end{equation}
For $r>0$, we define the \emph{local diffeomorphisms} $ \tilde\Phi_j: B(0,r)\cap \R^{2d} \to \R^{2d}$ by
$$ \tilde\Phi_j = P_{j+1}^{-1}\circ \Phi_j \circ P_j|B(0, r).$$

The local diffeomorphisms satisfy the following properties:
\begin{itemize}
	\item $\tilde\Phi_j(0,0)=(0,0)$.
	\item $D\tilde\Phi_j(0,0)=\diag\{M_j, N_j\}$. 
	\item Since $\|\Phi_j\|_{1+ \alpha}\le K_j$ and $\|Q_j\|, \|Q_j^{-1}\|\le K_ja_j^{-\frac12}$, we have $\|\tilde\Phi_j\|_{1+\alpha} \le K_j^3 a_j^{-1}$. 
\end{itemize}
This setting is the most general setting on which the Hadamard-Perron theorem can be established. 

Given $\rho_j>0$ to be chosen later, we denote $\sigma_j = \|D\tilde\Phi_j(u,s)-D\tilde\Phi_j(0,0)\|_{B(0, \rho_j)}$. For $\gamma>0$, we define the unstable and stable  cone field by 
$$ C^u_{\gamma}=\{(s,u)\in \R^{2d}; \|s\|\le \gamma \|u\|\}, \quad C^s_{\gamma}=\{(s,u)\in \R^{2d}; \|u\|\le \gamma \|s\|\}.$$
We say a set $W\subset \R^{2d}$ is $(u, \gamma, \rho_j, j)-$admissible if there exists a $\gamma-$Lipshitz map $\varphi: \R^d\cap B(0,\rho_j) \to \R^d$ such that $\varphi(0)=0$ and $ W=\{ (\varphi(u), u)\}$. A set $W\subset \R^{2d}$ is $(s, \gamma, \rho_j, j)-$admissible if there exists a $\gamma-$Lipshitz map $\varphi: \R^d\cap B(0,\rho_j) \to \R^d$ such that $\varphi(0)=0$ and $W=\{(s, \varphi(s))\}$. Denote $D_\rho = \{(s,u) \in \R^{2d}; \|s\| \le \rho, \|u\|\le \rho\}$.

\begin{proposition}\label{graph-trans}
Assume that for $j\le 0$, the parameters $\rho_j$ and $\sigma_j$ satisfy the following conditions:
\begin{itemize}
	\item $e^{-\lambda_j}+2\sigma_j <e^{-\lambda_j/2}$,
	\item $e^{\lambda_j} -2\sigma_j > e^{\lambda_j/2}$,
	\item $e^{\lambda_j/2} \rho_{j}\ge \rho_{j+1}$.
\end{itemize}
Then the following hold:
\begin{enumerate}
	\item For any $(s,u)\in D_{\gamma_j, \rho_j}$, we have $D\tilde\Phi_j(s,u)C^u_1 \subset C^u_{\gamma_j}$, where $\gamma_j = e^{-\lambda_j/2}$. 
	\item If  $W$ is a $(u,1, \rho_j, j)-$admissible set, then $\cG(W) :=\tilde\Phi_j(W)\cap D_{\rho_{j+1}} $ is a $(\gamma_j, \rho_{j+1}, j+1)-$admissible set. Furthermore, for $(s_j, u_j)\in W$, let $(s_{j+1},u_{j+1})=\tilde\Phi_j(s_j, u_j)$, we have 
	$$ \|u_{j+1}\| \ge e^{\lambda_j/2} \|u_j\|. $$
	\item If  $W$ is a $(s,1, \rho_j, j)-$admissible set, then $\cG^{-1}(W) :=\tilde\Phi_{j-1}^{-1}(W)\cap D_{\rho_{j-1}} $ is a $(\gamma_{j-1}, \rho_{j-1}, j-1)-$admissible set. Furthermore, for $(s_j, u_j)\in W$, let $(u_{j-1},s_{j-1})=\tilde\Phi_{j-1}^{-1}(s_j, u_j)$, we have 
	$$ \|s_{j+1}\| \ge e^{\lambda_j/2} \|s_j\|. $$
	\item Let $(W_j)_{j=-\infty}^0$ be a sequence of $(u, 1, \rho_j, j)-$admissible sets, then the map $\tilde\Phi_j$ induces a map $\cG(W)_{j+1}:=\cG(W_{j-1})$ on the set of such sequences. We have that this sequence is a contraction in Lipshitz norm, and there exists a unique limit sequence. Furthermore, this limit sequence is $C^1$.
\end{enumerate}
\end{proposition}

\begin{proof}
	Let $(\delta s,\delta u)\in \R^{2d}$  satisfy $\|\delta s\|\le \|\delta u\|$, and let $(\delta s',\delta u')=D\Phi_j(s,u) (\delta s, \delta u)$. We have 
	\begin{equation}
		\label{eq:expansion}
		\|\delta  s'\| \le e^{-\lambda_j}\|\delta  s\| + \sigma_j \|(\delta s, \delta u)\|, \quad \|\delta  u'\| \ge e^{\lambda_j}\|\delta  u\| - \sigma_j \|(\delta s, \delta u)\|.
	\end{equation}  
	Using $\|\delta s\|\le \|\delta u\|$, we obtain $\|\delta s'\| \le \gamma_j\|\delta u'\|$, where $\gamma_j = \frac{e^{-\lambda_j}+2\sigma}{e^{\lambda_j}-2\sigma}$. This proves the first statement of our proposition. 

	To show the image of a unstable admissible manifold is still admissible,  the calculation is similar and we refer to Proposition 7.3.5 of \cite{BP07}. The expansion in $u$ component is a consequence of \eqref{eq:expansion}. This proves the second statement. The third statement follows from a symmetric argument. 

	For the fourth statement, by Proposition 7.3.6 of \cite{BP07}, the map $\cG (W_j)$ is a contraction with factor $e^{-\lambda_{j-1}/2}$ at the $j-$th component. By Lemma~\ref{slowvarying}, we know that $\prod_{k=-n}^{-1} e^{-\lambda_{j+k}/2} \le \kappa_j^{-1} e^{n(-\lambda+\epsilon)}$. For $\epsilon <\lambda$, there exists $n>0$ such that all   $\prod_{k=-n}^{-1} e^{-\lambda_{j+k}/2}<1$. This implies that $\cG^n$ is a uniform contraction. 
\end{proof}

We can choose  parameters such that the conditions of Proposition~\ref{graph-trans} is satisfied. First, we have the following lemma:
\begin{lemma}\label{rhoj}
Given a positive function $\rho_0:\Omega\to \R$, we define $\rho_j(\omega) = \rho_0 \prod_{k=j}^{-1} e^{-\lambda_j/4}$ for all $j\le -1$. Then for any positive function $R:\Omega \to \R$ such that $e^{-\epsilon} \le R\circ \theta/R \le e^{\epsilon}$, there exists $\rho_0$ such that 
$$ \rho_j \le R\circ \theta^j, \, j \le -1 \text{ and } e^{-2\epsilon} \le \frac{\rho_0\circ \theta}{\rho_0} \le e^{2\epsilon}. $$
\end{lemma}
\begin{proof}
	By Lemma~\ref{slowvarying}, 
	$$ \rho_j = \rho_0 \prod_{k=j}^{-1} e^{-\lambda_j/4} \le \rho_0 \kappa_0^{\frac14} e^{-\frac{(\lambda-\epsilon)}{4}|j|}.$$
	Take $\rho_0 = R \kappa^{-\frac14}$, we have that for $\lambda > 2\epsilon$,  
	$$ \rho_j \le R e^{-(\lambda-\epsilon)|j|} \le R e^{-\epsilon |j|} \le R\circ \theta^j. $$
\end{proof}

\begin{proposition}\label{parameters}
There exists $\rho_0=\rho_0(\omega)>0$ satisfying 
$$ e^{-\epsilon}\le \frac{\rho_0(\theta\omega)}{\rho_0(\theta)} \le e^{\epsilon},$$
such that for $\rho_j(\omega) = \rho_0 \prod_{k=j}^{-1} e^{-\lambda_j/4}$, the conditions of Proposition~\ref{graph-trans} are satisfied for all $j\le 0$. 
\end{proposition}
\begin{proof}
	It's easy to see that the first 2 conditions of Proposition~\ref{graph-trans} are satisfied if $\sigma_j < \lambda_j/8$. Since $ \sigma_j \le \|\tilde\Phi_j\|_{1+ \alpha} \rho_j^\alpha= K_j^3 a_j^{-1}\rho_j^\alpha$ and $ a_j/2K_j \le \lambda_j$, it suffice to have 
	$$ \rho_j^\alpha \le \frac{1}{20} \frac{a_j^2}{K_j^4}, \quad j \le 0. $$
	Let $R = \frac{1}{20} \left(\frac{a_0^2}{K_0^4}\right)^{1/\alpha}$, by re-choosing the functions $a_0$, $K_0$ using Lemma~\ref{slowvarying} and parameter $\epsilon/(12\alpha)$, we can guarantee $e^{-\epsilon/2} \le R\circ \theta/R \le e^{\epsilon/2}$. By Lemma~\ref{rhoj}, there exists $\rho_0$ such that $\rho_j \le R\circ \theta^j$ and $e^{-\epsilon} \le \rho_0\circ \theta/\rho_0 \le e^{\epsilon}$. 
\end{proof}

\section{Local smoothness of the viscosity solutions}
\label{sec:smoothness}

Roughly speaking, our proof of Theorem~\ref{smooth} is based on the following observation: near a hyperbolic orbit, any orbits that does not deviate exponentially in the backward time from the given hyperbolic orbit must be contained in the unstable manifold. For $j\le 0$, we denote $W_j^-=\{(x, \nabla_x \psi^-(x,j)\}$ and $\tilde{W}_j = P_j^{-1} W_j^-$. We will first show that the orbits contained in $\tilde{W}_j$ does not ``expand exponentially'' in backward time, with a technical assumption that, at the last iterate, the orbit's unstable component is dominated by its stable component.  

\begin{proposition}\label{small-expansion}
The manifolds $\tilde{W}_j$ has the following properties:
\begin{enumerate}
	\item $\tilde{W}_j$ is a family of invariant sets for the local diffeormorphisms $\tilde\Phi_j$, in that 
	$$ \tilde{W}_{j+1} \subset \tilde\Phi_j\tilde{W}_j. $$
	\item There exists $\tilde{C}_j>0$ and $R_j >0$ satisfying 
	$$e^{-\epsilon} \le \frac{\tilde{C}_{j+1}}{\tilde{C}_j}, \frac{R_{j+1}}{R_j}  \le e^{\epsilon} ,$$
	such that the following hold. Given $0<\gamma <1$ and $k\ge 0$, let  $(s_j, u_j) \in \tilde{W}_j$ be a backward orbit satisfying
	\begin{enumerate}
		\item $\|u_{j-i}\| + \|s_{j-i}\| \le R_{j-i}$ for all $0\le i \le k$,
		\item $\|u_{j-k}\| \le \gamma\|s_{j-k}\|$,
	\end{enumerate}	  
	then we have 
	\begin{equation}
		\label{eq:growth-bnd} 
		\|s_{j-k}\| \le (1-\gamma)^{-1} \tilde{C}_j e^{\epsilon k} \|s_j\| 
	\end{equation}
	for all $k \ge 0 $ such that $\|u_j\|' + \|s_j\|'\le R_j $. 
\end{enumerate}
\end{proposition}

The proof of Proposition~\ref{small-expansion} relies on the following properties of the viscosity solutions. Denote, for the rest of this section,
$$ \psi(x,j) = \psi^-(x,j) - \psi^+(x,j). $$

\begin{lemma}\label{dec-barrier}
For $(y_0,w_0)\in W_0^-(x_0,v_0)$, denote $(y_j,w_j)=\Phi_{j,0}^{-1}(y_0,w_0)$ for all $j<0$. Then 
$$ \psi(y_j, j) - \psi(x_j, j) \le \psi(y_0, 0) - \psi(x_0, 0).$$
\end{lemma}
\begin{proof}
	Since the orbits $(x_k,v_k)_{k\le 0}$ and $(y_k,w_k)_{k\le 0}$ are backward minimizers, we have 
	$$ \psi^-(x_0,0) = \psi^-(x_j,j) + A_{j,0}(x_j, x_0), \quad \psi^-(y_0,0) = \psi^-(y_j,j) + A_{j,0}(y_j, y_0).$$
	Furthermore, since $(x_k, v_k)_{k\ge j}$ is also a forward minimizer, we have
	$$ \psi^+(x_j,0) = \psi^+(x_0, 0) - A_{j,0}(x_j, x_0).$$
	It follows from the definition of $\psi^+$ that 
	$$ \psi^+(y_j, 0)\ge \psi^+(y_0, 0) - A_{j, 0}(y_j, y_0). $$
	Hence
	\begin{align*}
		& \psi(y_0, 0) -  \psi(x_0, 0) = \psi^-(y_0, 0) - \psi^+(y_0, 0) - \psi^-(x_0,0) + \psi^+(x_0,0) \\
		& = \psi^-(y_j, 0) - (\psi^+(y_0, 0) - A_{j,0}(y_j, y_0)) - \psi^-(x_j, 0) + (\psi^+(x_0,0) - A_{j,0}(x_j, x_0)) \\
		& \ge \psi^-(y_j, 0) - \psi^+(y_j, 0) -\psi^-(x_j,0) + \psi^+(x_j, 0).
	\end{align*}
\end{proof}

\begin{proposition}\label{strict-min}
There exists  $b(\omega)>0$ with
$$\int b(\omega)^{-\frac12}dP(\omega) < \infty $$
and a  constant $r(F)>0$ depending only on $(F_1,\cdots, F_M)$, such that 
$$ b(\omega) \|y-x_0\|^2 \le \psi(y, 0) - \psi(x_0, 0) \le (1+ C_0^\omega) \|y-x_0\|^2, \quad \text{ for }\|y-x_0\|\le r(F). $$
\end{proposition}
\begin{proof}
	The lower bound follows from Proposition~\ref{strict-nondeg}, and the upper bound is a consequence of the semi-concavity (Lemma~\ref{properties}). 
\end{proof}
Applying Lemma~\ref{tempering} to $b(\omega)$, we obtain function $0<b_0(\omega)\le b(\omega)$ with $e^{-\epsilon}\le b_0(\theta\omega)/b_0(\omega)\le e^\epsilon$. Write $b_j(\omega)=b_0(\theta^j \omega)$. Before we prove Proposition~\ref{small-expansion}, we need the following lemma, which links the variationally defined objects and those coming from hyperbolicity.
\begin{lemma}\label{stable-est}
Assume $0<\gamma<1$, let $(s,u)\in D_{\rho_j}\cap C^s_\gamma$, and write $(y,w)=P_j(s,u)$,  where $P_j$ is from \eqref{eq:Pj}. Then 
$$ (1-\gamma)K_j^{-\frac12}\|s\| \le \|y-x_j\| \le 2a_j^{-\frac12}\|s\|.$$
\end{lemma}
\begin{proof}
	We have 
	$$
	\begin{bmatrix}
		y - x_j \\ w- v_j
	\end{bmatrix} =
	\begin{bmatrix}
		I & I \\
		\bbS_j & \bbU_j
	\end{bmatrix}
	\begin{bmatrix}
		(U_j-S_j)^{-\frac12} s \\ (U_j-S_j)^{-\frac12} u
	\end{bmatrix},
	$$
	hence $y-x_j=(U_j-S_j)^{-\frac12} (s + u)$. It follows that 
	$$ \|y-x_j\| \le \|(U_j-S_j)^{-\frac12}\| \|s+u\| \le a_j^{-\frac12}(1+\gamma)\|s\| \le 2a_j^{-\frac12}\|s\|,$$
	$$ \|y-x_j\| \ge m((U_j-S_j)^{-\frac12}) \|s+u\| \ge K_j^{-\frac12}(1-\gamma)\|s\|.$$
\end{proof}

\begin{proof}[Proof of Proposition~\ref{small-expansion}]
	Let $R_j = K_j^{-1} a_j^{\frac12} r(F) \le \|P_j\|^{-1} r(F)$. Let $(s_j, u_j)\in \tilde{W}_j$ be a backward orbit satisfying the conditions (a) and (b) in Proposition~\ref{small-expansion} and let $(y_j, w_j) = P_j(s_j, u_j)$, we have $\|u_j\| + \|s_j\| \le R_j$ implies $\|y_j -x_j\| \le r(F)$. By Lemma~\ref{dec-barrier} and Proposition~\ref{strict-min}, we have that for $j \le 0$ and $k\ge 0$, 
	\begin{multline*} \frac12 b_{j-k} \|y_{j-k} -x_{j-k}\|^2 \le \psi(y_{j-k}, j-k) -  \psi(x_{j-k}, j-k)  \\
	\le  \psi(y_j, j) - \psi(x_j, j)   \le K_j \|y_j -x_j\|^2. \end{multline*}
	By Lemma~\ref{stable-est}, we have 
	$$\|y_j-x_j\| \le 2a^{-\frac12}\|s_j\| , \quad \|s_{j-k}\| \le (1-\gamma)^{-1} K_j^{\frac12}\|y_j-x_j\|.$$
	Combine all three inequalities, we have
	\begin{multline*}
		\|s_{j-k}\|\le \frac12 (1-\gamma)^{-1} K_j^{\frac12} b_{j-k}^{-1} \|y_{j-k} - x_{j-k} \| \\
		\le \frac12 (1-\gamma)^{-1} K_j^{\frac32} b_{j-k}^{-1} \|y_j -x_j\| \le (1-\gamma)^{-1} K_j^{\frac32} b_{j-k}^{-1} a_j^{-\frac12} \|s_j\|.
	\end{multline*}
	Let $\tilde{C}_j = K_j^{\frac32} b_{j}^{-1} a_j^{-\frac12}$, we obtain 
	$\|s_{j-k}\| \le (1-\gamma)^{-1}\tilde{C}_j e^{\epsilon k} \|s_j\|$. 
	By re-choosing the functions $K_0$, $b_0$, $a_0$ again if necessary, we also have 
	$ e^{-\epsilon} \le \tilde{C}_{j+1}/\tilde{C}_j \le e^{\epsilon} $.	
\end{proof}

We now prove Theorem~\ref{smooth} using Proposition~\ref{small-expansion}. 

\begin{proof}[Proof of Theorem~\ref{smooth}]
	The proof proceeds in two steps. In the first step, we show that any backward orbit satisfying the conditions of Proposition~\ref{small-expansion} must be contained in the unstable cone, i.e. $\{(s,u)\in \R^{2d}; \|s\|\le \|u\|\}$. In the second step, we show that any backward orbit that are contained in the unstable cone for every iterate must be contained in the unstable manifold. 

	\emph{Step one}. 
	By Lemma~\ref{rhoj}, we can choose $\rho_0$ satisfying conditions of Proposition~\ref{slowvarying} and $\rho_j \le R_j/2$. Define 
	$$\bar{r}_0 = \tilde{C}_0^{-3} K_0^{-9} a_0^6 \kappa_0^2 \rho_0^2, $$
	we claim that any $(s_0, u_0)\in \tilde{W}_0 \cap D_{\bar{r}_0}$ must satisfy 
	$$ \|s_0 \| \le \|u_0\|. $$
	It suffices to show that $\|u_0\| \le \|s_0\|$ and $\|s_0\| \le \rho_0$ implies $\|s_0\| \ge \bar{r}_0$. In this case,  $(s_0, u_0)$ is contained in a $(s, 1, \rho_0, 0)-$admissible set.  Let 
	$$m=\min\{i\ge 0: (s_{-i}, u_{-i}) \notin D_{\rho_{-i}}\}-1.$$
	By Proposition~\ref{graph-trans}, for all $0 \le i \le m$, $(s_{-i},u_{-i})$ is contained in a $(s, \gamma_{-i}, \rho_{-i}, -i)-$admissible set, hence $ \|u_{-i} \| \le \gamma_{-i}\|s_{-i}\|$. By the definition of $m$ and $\|\tilde\Phi_j\|\le K_j^3 a_j^{-1}$, we have 
	$$ K_{-m}^{-3} a_{-m} \rho_{-m} \le \|s_{-m}\| \le \rho_{-m}. $$ 
	Applying \eqref{eq:growth-bnd}, we have 
	$$K_{-m}^{-3} a_{-m} \rho_{-m} \le \|s_{-m}\| \le (1-\gamma_{-m})^{-1}\tilde{C}_0 e^{\epsilon k}\|s_0\|. $$
	We have $(1-\gamma_{-m})^{-1} = (1- e^{-\lambda_{-m}/2})^{-1} \le \lambda_{-m}^{-1} \le a_{-m}^{-1}$.	It follows that 
	\begin{equation}
		\label{eq:sj-bnd} 
		K_0^{-3} a_0^2  \le \tilde{C}_0 e^{5m\epsilon }\|s_0\| \rho_{-m}^{-1}. 	
	\end{equation}

	Furthermore, by Proposition~\ref{graph-trans}, part (3), and $\|s_{-m}\| \le \rho_{-m}$, we have
	$$ \prod_{i=-m}^{-1}e^{\lambda_i/2}\|s_0\| \le \|s_{-m}\|  \le \rho_{-m} = \rho_0 \prod_{i=-m}^{-1}e^{-\lambda_i/4} . $$
	It follows that $\|s_0\|\rho_0^2 \le \rho_k^3$, hence
	\begin{equation}
		\label{eq:rhok}
		\|s_0\|^{\frac13} \le \rho_0^{-\frac23} \rho_{-m}.
	\end{equation}	
	On the other hand, by Lemma~\ref{slowvarying}, 
	$$ \kappa_0^{\frac12} e^{(\lambda-\epsilon)m/2} \|s_0\| \le \rho_{-m} \le 1. $$
	We may choose $\epsilon$ sufficiently small such that $\frac{10\epsilon}{\lambda-\epsilon}<\frac13$. We have
	$$ e^{5\epsilon m} \le \kappa_0^{-\frac{5\epsilon}{\lambda-\epsilon}} \|s_0\|^{-\frac{10\epsilon}{\lambda-\epsilon}} \le \kappa_0^{-\frac23} \|s_0\|^{-\frac13}. $$
	By \eqref{eq:sj-bnd} and \eqref{eq:rhok}, we have 
	$$ K_0^{-3} a_0^2  \le \tilde{C}_0 e^{5m\epsilon }\|s_0\| \rho_{-m}^{-1} = 
	\tilde{C}_0 \kappa_0^{-\frac23} \|s_0\|^{\frac13} \rho_0^{-\frac23} \|s_0\|^{\frac13},$$
	hence 
	$$ \|s_0\| \ge \tilde{C}_0^{-3} K_0^{-9} a_0^6 \kappa_0^2 \rho_0^2. $$ 
	This concludes the first step. 

	\emph{Step two.} For $j\le 0$, define $\bar{r}_j(\omega)=\bar{r}_0(\theta^j \omega)$. We may apply step one to $\theta^j \omega$ instead of $\omega$, and obtain that for any $(s_j, u_j) \in \tilde{W}_j\cap D_{\bar{r}_j}$, $\|s_j\| \le \|u_j\|$. Define 
	$$ r_0 = \bar{r}_{-1} K_{-1}^{-3} a_{-1} \kappa_0^{\frac12}$$
	and $r_j = r_0 \prod_{k=j}^{-1} e^{-\lambda_k/2}$ for all $j\le -1$. Similar to Lemma~\ref{rhoj}, by choosing a small $\epsilon$, we have 
	$$ r_j \le \bar{r}_{j-1} K_{j-1}^{-3} a_{j-1}$$ 
	for all $j\le 0$. 

	For any $(s_0, u_0)\in D_{r_0}\cap \tilde{W}_0$, using $\|D\tilde\Phi_{-1}\|\le K_{-1}^3 a_{-1}^{-1}$, we have $(s_{-1}, u_{-1})\in D_{\bar{r}_{-1}}\cap \tilde{W}_{-1}$, and by step one $\|s_{-1}\| \le \|u_{-1}\|$. It follows that $(s_{_1}, u_{-1})$ is contained in a $(u, 1, \bar{r}_{-1}, -1)-$admissible set, and by Proposition~\ref{graph-trans}, 
	$$  \|u_{-1}\| \le e^{-\lambda_{-1}/2} \|u_0\| \le e^{-\lambda_{-1}/2} r_0 \le \bar{r}_{-2}K_{-2}^{-3} a_{-2}.$$
	This procedure can be continued indefinitely. It follows that, for all $j\le 0$,  $(s_{j}, u_{j})$ is contained in a $(u, 1, \bar{r}_j, j)-$admissible set.  Take a family $W_{j}$ of admissible manifolds that contain $(s_{j}, u_{j})$. The fact that $(s_j, u_j)$ is a backward orbit implies $(s_j, u_j)\in (\cG^n W)_j$ for all $n\ge 0$. Let $n\to \infty$, we conclude that $(s_j, u_j)\in \lim_{n\to \infty}\cG(W)_j = W^u_j$, the unstable manifolds.  

	The first part of Theorem~\ref{smooth} follows from taking $U(\omega)=P_0 D_{r_0}$. To prove the second statement, note that the subdifferential  $\partial_x\psi^-(x,0)$ is upper semi-continuous as a set function. Since $\nabla \psi^-(x_0,0)$ exists, there exists a neighbourhood $V(\omega)$ of $X_0$ such that for all $y\in V(\omega)$, $\partial_y \psi^-(y,0)\in U(\omega)$. Using the first part of the theorem, we have $(y,\nabla\psi^-(y, 0)) \in W^u_0$ for almost every $y\in V(\omega)$. Since $W^u$ is $C^1$, we conclude that $\nabla\psi^-$ is also $C^1$ and $\psi^-$ is $C^2$. 
\end{proof} 

\begin{remark}
	It is not difficult to deduce from Main Theorem~\ref{smooth-vis} that the graph $\{(x, \nabla \psi^-(x,0))\}$ is contained in a smooth manifold. More precisely, there exists $j<0$ such that $\{(x, \nabla \psi^-(x,0))\}$ is contained in the forward image of a smooth piece of $\{(x, \nabla \psi^-(x,j))\}$. This leads to interesting questions on the topology and regularity of the shock manifolds, however, we will not discuss it in this paper. 
\end{remark}

\section{Generic nondegeneracy of the minimum}
\label{sec:gen-nondeg}

In this section we prove Proposition~\ref{nondegenerate}. Recall that we have 
$$H(x,c)=\psi(x)+ V(x,c),$$
where $\psi$ is a semi-concave function and $V(x,c)$ is $C^2$. We assume that for every $c$, the map $\partial_c V(\cdot, c)$ is an embedding from $\T^d$ to $\R^M$. By compactness, 
\[
	K= \sup_{x, c}\|\partial^2_{22} V(x, c)\| < \infty. 
\]
Recall that $m(A)=\|A^{-1}\|^{-1}$ is the conorm of an $n\times n$ matrix $A$.

Using the definition of embedding, we immediately have the following statement. 
\begin{lemma}\label{chart}
Assume that $\partial_cV(\cdot, c):\T^d\to \R^M$ is an embedding. Then there exists $U_1, \cdots, U_k\subset \T^d$ such that $\T^d=\bigcup_{j=1}^k U_j$, and for each $U_j$, there exists a projection $\Pi_j:\R^M\to \R^d$ given by $\Pi_j(c_1, \cdots, c_M)=(c_{i_1}, \cdots, c_{i_d})$ for some indices $\{i_1, \cdots, i_d\}\subset \{1, \cdots, M\}$, and a continuous positive function $\mu_j:\R^M\to \R$ such that 
$$m(\partial_x(\Pi_j\circ \partial_cV)(x,c))\ge \mu_j(c), \quad \forall x\in U_j.$$
\end{lemma}

Write $G(c)=\inf_{x\in \T^d}H(x,c)$. We have the following results. 

\begin{lemma}\label{subdifferential}
$G$ is a $K-$semi-concave function. For any $c\in \R^M$, if $x(c)\in X(c):=\argmin_x \{\psi(x)+V(x,c)\}$, then $\partial_cV(x(c),c)$ is a superdifferential of  $G$ at $c$. Conversely, if $l_c$ is a superdifferential of $G$ at $c$, then $l_c\in \conv_{x\in X(c)} \{\partial_cV(x,c)\}$. Here $\conv$ denote the convex hull of a set.   
\end{lemma}
\begin{proof}
	For any $c', c\in \R^M$, let $x(c')\in X(c')$ and $x(c)\in X(c)$, we have that 
	\begin{multline*} G(c')-G(c)=H(x(c'),c')-H(x(c),c)\le H(x(c),c')-H(x(c),c)\\
	= V(x(c),c')-V(x(c), c) \le \langle \partial_cV(x(c),c), c'-c\rangle + \frac12\|\partial^2_{cc}V\| |c-c'|^2.
\end{multline*}
It follows that $G(c)$ is semi-concave and  $\partial_cV(x(c))$ is a superdifferential of $G(c)$ at $c$. 

For the converse, we first show that $X(c)$ as a set function is upper semi-continuous in the sense that if  $c_n\to c$, $x_n\in X(c_n)$ and $x_n\to x$, then $x\in X(c)$. Indeed, 
$$ G(c)=\lim_{n\to \infty}G(c_n)=\lim_{n\to \infty} \psi(x_n)+V(x_n, c_n) =\psi(x) + V(x,c),$$
hence $x\in X(c)$. 

We now argue by contradiction. Assume that $l_c \notin \conv_{x\in X(c)}\{\partial_cV(x,c)\}$, then there exists a vector $v\in \R^M$ such that $\langle l_c, v\rangle >\langle \partial_cV(x,c), v\rangle$ for all $x\in X(c)$. Then we have
$$ G(c-tv) \le G(c) -\langle l_c, tv\rangle + \frac12 t^2 K |v|^2,$$
$$ G(c) \le G(c-tv) + \langle \partial_cV(x_t,c-tv), tv\rangle + \frac12 t^2 K|v|^2,$$
where $x_t\in X(c-tv)$. Since the domain for $x$ is compact, there exists $t_n\to 0+$ such that $x_n:=x_{t_n}\to x\in X(c)$. Combine the two formulas, we have
$$ G(c-t_nv)\le G(c-t_nv) - \langle l_c, t_nv\rangle + \langle \partial_cV(x_n, c-t_nv),t_nv\rangle + t_n^2 K|v|^2.  $$
Divide both sides by $t_n$ and take $t_n\to 0+$, we obtain   a contradiction:
$$\langle l_c, v\rangle \le  \langle \partial_cV(x,c), v\rangle. $$

\end{proof}
\begin{corollary}\label{dG}(See Lemma~\ref{generic})
For almost every $c\in \R^M$, $G(c)$ is differentiable, there is a unique $x(c)\in \argmin_x H(x,c)$, and $\partial_cV(x(c),c)= d G(c)$. 
\end{corollary}
\begin{proof}
	Since a semi-concave function is almost everywhere differentiable, for almost every $c$, the superdifferential of $G(c)$ is unique. Since $\partial_cV(\cdot,c)$ is one-to-one,  there is at most one $x(c)$ such that $\partial_cV(x(c))=dG(c)$.
\end{proof}

\begin{lemma}
	Let $\varphi:\T^d\to \R$ be a semi-concave function and $x_0\in \argmin \varphi(x)$, then $\varphi$ is differentiable at $x_0$ and $d\varphi(x_0)=0$. 
\end{lemma}
\begin{proof}
	Let $l_0$ be a subdifferential of $\varphi$ at $x_0$, then for any $x\in \T^n$, 
	$$ 0\le \varphi(x)-\varphi(x_0) \le l_0(x-x_0) + C|x-x_0|^2.$$
	Take $x-x_0=tv$ for $t>0$, $v\in \R^n$, we have $0\le tl_0(v) + Ct^2|v|^2$. Divide by $t$ and let $t\to 0+$, we have that $l_0(v)\ge 0$ for any $v\in \R^n$, and hence $l=0$.  It follows that $\varphi$  has only $0$ as a subdifferential at $x_0$. As a consequence, $\varphi$ is differentiable at $x_0$ and the derivative is $0$. 
\end{proof}

The following corollary is an immediate consequence of the last lemma. 
\begin{corollary}\label{dxV}
Any $x(c)\in \argmin_xH(x,c)$ satisfies $\partial_x H(x(c),c)=0$. 
\end{corollary}

Given any function $g:\T^n\to \R$ differentiable at $x$, we define the second order difference to be
$$\nabla^2g(x)(\Delta x)=g(x+\Delta x)-g(x) - \langle dg(x), \Delta x\rangle. $$
We have
$$ d^2g(x)(v)=\liminf_{\tau\to 0+}\frac{2}{\tau^2}\nabla^2g(x)(\tau v). $$

Given any $x\in \T^d$ and $c\in \R^M$, we call $(x,c)$ a conjugate pair if $x$ is the unique point of minimum of $H(\cdot, c)$. This implies $\partial_c V(x,c)= DG(c)$, $\partial_x H(x,c)=0$ and  $G(c)=H(x,c)$ (in particular, all derivatives must exist). 
There is a duality between the second subderivative of $V$ in $x$ and the second subderivative of $G$. 
\begin{lemma}[Duality]\label{duality}
For any $\Delta x\in \R^d$ and $\Delta c\in \R^M$,
\begin{equation}
	\label{eq:diff}
	\nabla_x^2H(x,c)(\Delta x) \ge \nabla^2 G(c)(\Delta c) -\langle \partial_cV(x+\Delta x, c) - \partial_c V(x,c), \Delta c\rangle -K(\Delta c)^2/2,
\end{equation}
where $K = \sup_{x, c}\|\partial^2_{22}V(x, c)\|$. 
As a result, for $v \in \R^d$ and $w \in \R^M$,
$$ \frac12 d^2_x H(x,c)(v)\ge \frac12 d^2G(c)(w) - \langle \partial^2_{cx}V(x,c)v,w\rangle - \frac{K}{2}|v|^2. $$
\end{lemma}
\begin{proof}
	Since $(x,c)$ is a conjugate pair, $G(c)=H(x,c)$, $\partial_xH(x,c)=0$ and $\partial_cV(x,c)=DG(c)$. 
	\[
		\begin{aligned}
			& \nabla^2_xH(x,c)(\Delta x)  = H(x+\Delta x, c)- H(x,c) - \langle \partial_x H(x,c), \Delta x \rangle   = H(x+\Delta x, c)- H(x,c) \\
			&= H(x+\Delta x, c+ \Delta c)- H(x,c) -((H(x+\Delta x,c+\Delta c)-H(x+\Delta x),c)) \\
			& \ge G(c+\Delta c)-G(c) - \langle \partial_c V(x+\Delta x,c), \Delta c\rangle -K(\Delta c)^2/2 \\
			& = G(c+\Delta c)-G(c) -\langle \partial_cV(x,c),\Delta c\rangle -\langle \partial_c V(x+\Delta x,c)-\partial_cV(x,c), \Delta c\rangle - K(\Delta c)^2/2\\
			& = \nabla^2 G(c)(\Delta c) - \langle \partial_c V(x+\Delta x,c)-\partial_cV(x,c), \Delta c\rangle -K(\Delta c)^2/2.
		\end{aligned}
	\]
	The second formula follows from taking $\Delta x= \tau v$, $\Delta c=\tau w$,  divide by $\tau^2$, and take limit as $\tau\to 0+$.
\end{proof}

To finally prove Proposition~\ref{nondegenerate}, we need the following result from convex analysis. 
\begin{theorem}[Alexandrov Theorem]\label{alexandrov}\cite{Ale39}
Let $f: \R^M \to \R$ be a convex function, then for almost every $x \in \R^M$, $f$ is differentiable at $x$, and there exists an $M \times M$ symmetric matrix $A$ such that 
\[
	\phi(x + v) = f(x) + \langle \nabla f(x) , v\rangle + \frac12 \langle A v, v\rangle + o(|v|^2)
\]
as $v \to 0$. 
\end{theorem}
For any $c_0$, a $K$-semi-concave function $G(c)$ can always be made concave by subtracting the quadratic function $K|c-c_0|^2$. It follows that a semi-concave function is also twice differentiable almost everywhere.

\begin{proof}[Proof of Proposition~\ref{nondegenerate}]
	According to Corollaries~\ref{dG} and \ref{dxV}, for almost every $c\in \R^m$, $G$ is differentiable at $c$, there exists unique $x(c)\in \argmin_x H(x,c)$, $\partial_cV(x,c)=dG(c)$ and $\partial_x H(x(c),c)=0$. In other words, $(x(c),c)$ is a conjugate pair. Furthermore, by Theorem~\ref{alexandrov}, $d^2G(c)$ exists and is a symmetric bilinear form. Assume that $d^2G(c)(w)=\langle A(c) w, w\rangle$, we have that 
	$$ \sup_{\|w\|=1}|d^2G(c)(w)|= \|A(c)\|.$$

	There exists $1\le j\le k$ such that $x(c)\in U_j$. By Lemma~\ref{duality}, we have that for any $v\in \R^d$, $w\in \R^M$,
	$$ d^2_xH(x(c),c)(v) \ge  d^2G(c)(w) - 2\langle \partial^2_{cx} V(x(c),c)v, w \rangle - K|v|^2. $$
	By Lemma~\ref{chart}, there is a coordinate projection $\Pi_j:\R^M\to \R^d$ given by $(c_1, \cdots, c_M)\mapsto (c_{i_1}, \cdots, c_{i_d})$ for some indices $\{i_1, \cdots, i_d\}\subset \{1, \cdots, M\}$. Let  $\Pi_j':\R^M \to \R^{M-d}$ be the map to the complementary indices. We have that for any two vectors $w_1, w_2\in \R^M$,
	$$ \langle w_1, w_2\rangle = \langle \Pi_jw_1,\Pi_jw_2\rangle + \langle \Pi_j'w_1, \Pi_j' w_2\rangle.$$

	Choose $w$ such that $\Pi_jw=- t \,  \Pi_j \partial^2_{xc}V(x(c),c)v$ and $\Pi_j'w=0$, where  $ t>0$ is a parameter. We have
	$$ -\langle \partial^2_{xc}V(x(c),c)v, w \rangle =  t \, \langle \Pi_j\partial^2_{xc}V(x(c),c)v, \Pi_j\partial^2_{xc}V(x(c),c)v\rangle =  t \, |\Pi_j \partial^2_{xc}V(x(c),c)v|^2.$$
	Let $\nu(c)=\sup_{x\in \T^d}\|\partial^2_{xc}V(x,c)\|$, then $\mu_j(c)|v|\le |\Pi_j \partial^2_{cs}V(x)v|\le \nu(c)|v|$. 
	It follows that 
	$$ d^2_xH(x(c),c)(v) \ge - (\|A(c)\|+ K )M(c)^2 |v|^2   t^2 + 2 \mu_j(c)^2 |v|^2  t. $$
	Choosing
	$$ t =  t_j(c): = \frac{\mu_j(c)^2}{ (\|A(c)\|+K)\nu^2(c)},$$
	we have that $$d^2_x H(x(c),c)(v)\ge \mu_j^2(c) t_j(c)|v|^2.$$
	We now choose 
	\begin{equation}
		\label{quan-bound} 
		a(c)  = \frac{(\inf_{1\le j\le k}\mu_j(c))^4}{\nu^2(c)}(\|A(c)\|+K)^{-1}
	\end{equation}
	and the proposition follows.
\end{proof}

\section{A quantitative Alexandrov theorem}
\label{sec:alexandrov}

In this section we prove Proposition~\ref{integrability}.  

For $V(x,c)=-\sum_{i=1}^M c_i F_i(x)$, $\partial_cV=(F_1, \cdots, F_M)$ is independent of $c$. It follows that we can choose $\mu_j(c)=\mu_j$ and $\nu(c)=\nu$ independent of $c$. Furthermore, the constant $K=0$ in  (\ref{quan-bound}). We have
$a(c)= \alpha \|A(c)\|^{-1}$, where $\alpha = \frac{(\inf_{1\le j\le k}\mu_j)^4}{\nu^2}$. 

By Lemma~\ref{subdifferential}, any subdifferential of $G(c)$ is contained in the set $\conv_{x\in X(c)}\{\partial_cV(x,c)\}$, a subset of  $B:=\conv_{x\in \T^d}\{(F_1, \cdots, F_M)(x)\}$. Since $\bigcup_{x\in \T^d}(F_1, \cdots, F_M)(x)$ is a compact set, so is $B$. 

To prove Proposition~\ref{integrability}, it suffices to show that, for a density $\rho$ satisfying assumption 5, there exists $A_1(F)$ such that 
$$ \int \|A(c)\| \rho(c) dc < A_1(F).$$
The above formula follows from the following ``quantitative Alexandrov theorem''. 
\begin{theorem}\label{quant-alex}
Assume $f:\R^M\to \R$ is a convex function such that there exists a bounded set $B$ satisfying the following condition: for any $c\in \R^n$, we have $\partial f(c) \subset B$. Here $\partial f(c)$ denote the set of all subdifferential of $f$ at $c$.

Let $\rho: \R^M\to \R^+$ be a probability density satisfying assumption 5, and let $A(c)$ denote the hessian matrix of $f$ at $c$, which exists almost everywhere by Theorem~\ref{alexandrov}. We have
$$ \int \|A(c)\| \rho(c) dc < A_1(B). $$
\end{theorem}

We first prove a lemma for one-dimensional functions. 
\begin{lemma}\label{1-d-alex}
Assume that $g:\R\to \R$ is a convex function, $\partial g(t)\subset (a,b)$, then 
$$ \int g''(t) dt  \le b-a. $$
\end{lemma}
\begin{proof}
	For a one dimensional convex function, $g'(t)$ exists almost everywhere and is increasing. Without loss of generality, we may assume that $g'(t)$ is increasing. It follows from well known properties of monotone functions that for any $N>0$, 
	$$ \int_{-N}^N g''(t) dx \le  \lim_{N\to \infty}g'(N)-g'(-N)\le b-a. $$
\end{proof}

\begin{proof}[Proof of Theorem~\ref{quant-alex}]

	Given a positive semi-definite symmetric matrix $A$, let $v$ be a unit eigenvector of its largest eigenvalue. We have  $\|A^{\frac12}v\|=\|A^{\frac12}\|=\|A\|^{\frac12}$. Write $v=a_1e_1 + \cdots + a_Me_M$, we say that $A$ is $i$-positive if $|a_i| = \max_{1\le k \le M}|a_k|$. There may exist multiple  $i$'s such that $A$ is $k$-positive. 

	It follows from the maximality of $|a_i|$ that $|a_i|\ge \frac{1}{\sqrt{M}}$. Denote $w=a_ie_i$, we have  $|v-w|^2=1-a_i^2\,$,  and 
	$$ |A^{\frac12}w| \ge |A^{\frac12}v|- |A^{\frac12}(v-w)| \ge \|A\|^{\frac12}\left(1-\sqrt{1-a_i^2}\right). $$
	It follows from Lemma~\ref{1-d-alex} that
	$$ |A^{\frac12}e_i| \ge \frac{\left(1-\sqrt{1-a_i^2}\right)}{a_i^2} \|A\|^{\frac12} \ge \beta_M \|A\|^{\frac12},$$
	where $\beta_M>0$ is a constant depending only on $M$. 

	For each $1\le i \le n$, define a function $\varphi_i: \R^n\to \R$ by 
	$$ \varphi_i(c)=
	\begin{cases} 
	\|D^2f(c)\|, & D^2f(c) \text{ exists and is i-positive;} \\
	0, & \text{otherwise.}
\end{cases} $$
We have $\|D^2f(c)\| \le \sum_{i=1}^M \varphi_i(c)$ for almost every $c$ (since $D^2f(c)$ exists for almost every $c$ by Alexandrov's theorem).  

Since $f(c_1, \cdots, c_M)$ considered as a function of $c_i$ (with $\hat{c}_i$ as parameters) is convex, and $\partial^2_{c_i^2}f(x)=\langle A(c) e_i, e_i\rangle$, by Lemma~\ref{1-d-alex}, 
$$ \int |A(c)^{\frac12}e_i|^2dc =  \int \langle A(c) e_i, e_i\rangle dc\le \diam(B). $$
We have 
$$ \int \varphi_i(c)dc_i \le \int \|A(c)\| dc_i \le \frac{1}{\beta_M^2}\int \langle A(c) e_i, e_i\rangle dc\le \diam(B)/\beta_M^2.$$
Apply Fubini theorem, we have 
$$ \int \varphi_i(c)\rho(c)dx \le \int \hat{\rho}(\hat{c}_i)d\hat{c}_i \int \varphi_i(c)\rho_i(c_i) dc_i    \le \|\rho_i\|_{L^\infty} \diam(B)/\beta_M^2 \int \hat\rho_i(\hat{c}_i)d\hat{c}_i.$$
Define $A(F)=\sum_{1\le i \le M}\|\rho_i\|_{L^\infty}\|\hat{\rho}_i\|_{L^1} \diam(B)/\beta_n^2$ and the theorem follows.
\end{proof}

\section{Uniform nondegeneracy of the minimum}
\label{sec:uniform}

We prove Proposition~\ref{strict-nondeg} in this section. Instead of looking at the limiting behaviour, we attempt to get quantitative estimates of the second order difference $\nabla_xH(x,c)$ instead of the subderivative. 

For $V(x,c)=\sum_{i=1}^Mc_iF_i(x)$, we have  $K = \sup \|\partial^2_{22}V\| =0$ and $\partial_cV=DF$, where $F=(F_1, \cdots, F_M)$. We obtain the simplified \eqref{eq:diff}:
\begin{equation}
	\label{eq:nabla} 
	\nabla_x^2 H(x,c)(\Delta x) \ge \nabla^2 G(c)(\Delta c) - \langle DF(x)\Delta x, \Delta c\rangle + C(F) \|\Delta x\|^2 \|\Delta c\|,
\end{equation}
where  $C(F)=\|F\|_{C^2}$. 

We have the following counterpart of Theorem~\ref{quant-alex}:
\begin{lemma}\label{max-est}
Assume that $f:\R^M\to \R$ is a convex function such that all subdifferentials are contained in a bounded set $B$. Let $\rho$ be a density satisfying Assumption 5. Then for $i=1, \cdots, M$, there exist positive measurable functions $g_i:\R^M\to \R^+$ and a constant $C$ depending only on $\rho$ and $B$ satisfying 
$$\int \sqrt{g_i(c)}d\rho(c) < C,$$
and for every $c$ where $Df(c)$ exists, we have
$$ \nabla^2 f(c)(te_i) \le   g_i(c) t^2.$$
\end{lemma}
Before proving Lemma~\ref{max-est}, let us prove a simple statement about non-decreasing functions of one variable: 
\begin{lemma}\label{1d-max-function}
Let $h:\R\to \R$ be a right continuous non-decreasing function satisfying $h(b)-h(a)\le B$ for all $a<b$. Let
$$ g(t):=\sup_{r>0}\left\{\frac{1}{2r}(h(t+r)-h(t-r))\right\}=\sup_{r>0}\left\{\frac{1}{2r}\int_{t-r}^{t+r} dh\right\},$$
where $dh$ is the Stieltjes integral associated to $h$. Let $|\cdot|$ denote the one-dimensional Lebesgue measure, there exists a constant $C'$ such that
$$ |\{s:g(s)\ge t\}| \le C'B/t.$$
\end{lemma}
\begin{proof}
	The proof is similar to the standard proof of the weak $(1,1)-$inequality for the Hardy-Littlewood maximal principle. We have that for any $s\in \{g(s)\ge t\}$, there exists $I_s=(s-r_s, s+r_s)$ such that $t|I_s|\le \int_{I_s}dh$. Since $\bigcup_{s}I_s \supset \{g(s)\ge t\}$, using the Besicovitch covering lemma, there exists a subcover $\{I_i\}$ with multiplicity at most $C'$. We have 
	$$ |\{s: g(s)\ge t\}| \le \frac{1}{t}\sum_{i}\int_{I_i}dh \le \frac{C'}{t} \int_{\R}dh \le \frac{C'B}{t}. $$
\end{proof}

\begin{proof}[Proof of Lemma~\ref{max-est}]
	Since $f(c+te_i)$ is a single variable convex function,  its first derivative $h_{c,i}(t):=\langle df(c+t e_i), e_i\rangle$ is defined almost everywhere and is increasing. We extend $h_{c,i}$  to be a right continuous function defined everywhere. By Rademacher's theorem, $f(c+be_i)-f(c+ae_i)=\int_a^b h_{c,i}(t)dt$. Moreover, by our assumption, 
	$$ h_{c,i}(b)-h_{c,i}(a) \le \diam (B), \quad \text{ for any } a<b,$$
	where $B$ is the compact set that contains all subdifferentials of $f$. 

	Define 
	$$\bar{g}_{i}(c)=\sup_{r > 0}\left\{ \frac{1}{2r}(h_{c,i}(r)-h_{c,i}(-r)) \right\}.$$ By Lemma~\ref{1d-max-function}, there exists $C'>0$ such that for all $t>0$,
	$$     |\{s: \bar{g}_i(c+se_i)>t\}|\le C' \diam(B) /t ,     $$
	where $|\cdot|$ stands for the one-dimensional Lebesgue measure. We have 

	\begin{multline*}\int_{\{s: \bar{g}_i(c+se_i)\ge 1\}} \sqrt{\bar{g}_i(c+se_i)} \, ds = \int_1^\infty|\{t: \sqrt{\bar{g}_i(c+se_i)}>t\}|dt \\
	\le C'\diam(B)\int_1^\infty t^{-2}dt \le C'\diam(B).
\end{multline*}
Define $g_i(c) = \max\{g_i(c),1\}$, we have 
\begin{multline*} \int \sqrt{g_i(c)}d\rho(c) 
\le 1 + \int_{\{\bar{g}_i\ge 1\}} \sqrt{\bar{g}_i(c)}d\rho(c) 
\le 1+ \int_{\R^{M-1}} \int_{ \{\bar{g}_i\ge 1\}} \sqrt{\bar{g}_i(c)}d\rho_i(c_i) d\hat{\rho}_i(\hat{c}_i)  \\
\le 1+ C'\diam{B}\|\rho_i\|_{L^\infty} \|\hat{\rho}_i\|_{L^1}:=C.
\end{multline*}

Assume that $Df(c)$ exists. We have
\begin{align*}
	\nabla^2 f(c)(te_i) & = f(c+te_i) - f(c) - \langle Df(c), te_i\rangle = \int_0^t\langle\, Df(c+se_i)- Df(c), e_i \rangle \, ds \\
	&  = \int_0^t (h_{c,i}(s) - h_{c,i}(0)) ds  \le \int_0^t 2s g_i(c) ds \le g_i(c)t^2.  
\end{align*}
\end{proof}

\begin{proof}[Proof of Proposition~\ref{strict-nondeg}]
	The proof is similar to that of Proposition~\ref{nondegenerate}. Let $x\in \argmin H(x,c)$. Using Lemma~\ref{chart}, there exists a projection $\Pi_j$ such that $m(\Pi_jDF(x))\ge \mu_j>0$.  Let $h=-\lambda \Pi_j(DF(x)\Delta x)$, where $\lambda$ is a parameter to be chosen later. We have $\|h\| \le \lambda \nu \|\Delta x\|$, where $\nu  = \max \|DF(x)\|$, and 
	$$ \langle DF(x) \Delta x, -h\rangle \ge \lambda \|\Pi_j DF(x) \Delta x\|^2. $$
	We make a simple observation about inner products in $\R^M$. Let $a, b\in \R^M$ satisfy $\langle a, b\rangle >0$. Assume that $b=\sum t_k e_k$, choose $i$ such that  $t_i \langle a_i, e_i\rangle = \max_k \{t_k \langle a_k, e_k\rangle\}$. Then $\langle a, t_i e_i \rangle \ge \frac{1}{M}\langle a, b\rangle$. 

	Apply the above observation to  $DF(x) \Delta x$ and $-h$, we obtain that there exists $1\le i \le M$ and a vector $te_i$ with $\|te_i\|\le \|h\|$ and 
	$$ \langle DF \Delta x, te_i \rangle \ge \frac{\lambda}{M} \|\Pi_j DF(x) \Delta x\|^2 \ge \frac{\lambda \mu_j^2}{M} \|\Delta x\|^2. $$
	Denote $\Delta c = -te_i$. Since $G(c)$ is concave, and $DG(c)$ exists, we can apply Lemma~\ref{max-est} to $-G$ at $c$.  Then there exists functions $g_i(c)$ such that $\nabla G(c)(\Delta c)\ge - g_i(c) \|\Delta c\|^2$.
	By (\ref{eq:nabla}), we have 
	\begin{align*}
		& \nabla_x^2H(x,c)(\Delta x) \ge \nabla^2 G(c)(\Delta c) -\langle DF(x)\Delta x, \Delta c\rangle -C(F)\|\Delta x\|^2\|\Delta c\|  \\
		& \ge -g_i(c)\|\Delta c\|^2  + \frac{\lambda \mu_j^2}{M}\|\Delta x\|^2 -C(F) \|\Delta x\|^2 \|\Delta c\| \\
		& \ge \left( -g_i(c) \lambda^2 \nu^2  + \frac{\lambda \mu_j^2}{M} - C(F) \lambda \nu \|\Delta x\| \right) \|\Delta x\|^2. 
	\end{align*}
	Let $\lambda = \frac{\mu_j^2}{4M g_i(c)\nu^2}$ and assume $\|\Delta x\| \le \min_j\frac{\mu_j^2}{4M C(F) \nu}=:r(F)$, we obtain that 
	$$ \nabla_x^2 H(x,c)(\Delta x)\ge \frac{\lambda \mu_j^2}{4M} = \frac{\mu_j^4}{16M^2 g_i(c)\nu^2}, \quad \|\Delta x\|\le r(F).  $$
	Define 
	$$\bara(c) = \frac{\min_j\mu_j^4}{8M^2 \max_i\{g_i(c)\}\nu^2},$$
	we have $\nabla_x H(x,c)(\Delta x)\ge \frac12 \bara(c)\|\Delta x\|^2$. 

	It remains to prove the integrability property. Since  $\bara(c)= C''/ \max_i\{g_i(c)\}$, where $C''= (\min \mu_j^4)/(8M^2\nu)$, we have 
	$$ \int \bara(c)^{-\frac12}d\rho(c) \le \sqrt{C''}\int \sum_{i=1}^Mg_i(c)^{\frac12} d\rho(c) \le MC\sqrt{C''}.  $$
\end{proof}

\bibliographystyle{plain}
\bibliography{HJ}
\end{document}